\pdfoutput=1
\RequirePackage{ifpdf}
\ifpdf % We~are running pdfTeX in pdf mode
\documentclass[pdftex]{sigma}
\else
\documentclass{sigma}
\fi

\usepackage{enumitem}
 \usepackage[all]{xy} % Include XY-pic
 \SelectTips{cm}{10} % Use the nicer arrowheads
 \everyxy={<2.5em,0em>:} % Sets the scale I like

\numberwithin{equation}{section}

\newtheorem{Theorem}{Theorem}[section]
\newtheorem*{Theorem*}{Theorem}
\newtheorem{Corollary}[Theorem]{Corollary}
\newtheorem{Lemma}[Theorem]{Lemma}
\newtheorem{Proposition}[Theorem]{Proposition}

\theoremstyle{definition}
\newtheorem{Definition}[Theorem]{Definition}

\newtheorem{Example}[Theorem]{Example}
\newtheorem{Remark}[Theorem]{Remark}

\begin{document}

\allowdisplaybreaks

\newcommand{\arXivNumber}{2307.07497}

\renewcommand{\PaperNumber}{110}

\FirstPageHeading

\ShortArticleName{The Geometric Deformation of Curved $L_\infty$ Algebras and Lie Algebroids}

\ArticleName{The Geometric Deformation\\ of Curved $\boldsymbol{L_\infty}$ Algebras and Lie Algebroids}

\Author{Xiaoyi CUI}

\AuthorNameForHeading{X.~Cui}

\Address{School of Mathematics (Zhuhai), Sun Yat-sen University, Zhuhai 519082, P.R.~China}
\Email{\mail{cuixy9@mail.sysu.edu.cn}}

\ArticleDates{Received November 26, 2025, in final form October 10, 2026; Published online November 09, 2026}

\Abstract{While $L_\infty$ algebras are fundamental structures in differential geometry and mathematical physics, the geometric information encoded in such structures is often implicit. We address the following question: What constitutes a~geometrically meaningful deformation of an $L_\infty$ algebra arising from vector bundles, and how can such deformations classify new geometric invariants? Inspired by nonabelian extension theory of Lie algebras, we define geometric deformations of curved $L_\infty$ algebras constructed from a~vector bundle $V\to M$, and demonstrate that such deformations uniquely correspond to Lie algebroid structures on~$V$. Explicit computations reveal that a~representative of the first Atiyah--Chern class is exact in the total Weil complex, with its primitive representing the modular class in Lie algebroid cohomology. In the case of action Lie algebroids, the leading-order Atiyah--Chern classes correspond to the equivariant Chern characters. Applications to BV theories show that the geometric deformations naturally generate Poisson sigma models. These results provide a~coherent framework for deriving field theories from geometric deformations of $L_\infty$ algebras.\looseness=-1 }


\Classification{53D17; 53Z05; 57R91}

\section{Introduction}

Given a~Lie algebra $\mathfrak k$ together with a~vector space $\mathfrak a$, it is natural to consider the extension problem, namely, to classify all possible Lie algebra structures on the linear direct sum $\mathfrak a\oplus \mathfrak k$ for which the canonical projection $\mathfrak a\oplus \mathfrak k\to \mathfrak k$ is a~Lie algebra homomorphism. Each such structure extends to a~short exact sequence of Lie algebras
$0\to \mathfrak a\to \mathfrak a\oplus \mathfrak k\to \mathfrak k\to 0$.
The classification of such structures is encoded in the (nonabelian) Lie algebra cohomology of~$\mathfrak k$, given a~fixed Lie module structure on $\mathfrak a$.

The coextension problem concerns the natural injection map $\mathfrak k\to \mathfrak a\oplus \mathfrak k$ and aims to determine the conditions under which this map defines a~Lie algebra homomorphism. Such a~coextension endows $\mathfrak a$ with a~$\mathfrak k$-module structure. However, unlike the extension problem, this construction does not necessarily yield a~short exact sequence and has received less attention. Nonetheless, this framework encompasses several interesting examples such as $r$-matrices and twilled Lie algebras~\cite{BGN, KSM}. The general problem is inherently complex, necessitating additional constraints to obtain meaningful classifications~\cite{AM}.

In the current paper, we investigate a~version of the coextension problem at the intersection of (curved) $L_\infty$ algebra and differential geometry. We fix a~real vector bundle $V$ over a~smooth manifold $M$. Throughout the paper, we shall let $\mathfrak h$ denote the $L_\infty$ algebra $\Omega_M\otimes \Gamma( T_M[-1])$ (see Proposition~\ref{prop:hdefn}) and let $\mathfrak v$ denote the $L_\infty$ module $\Omega_M\otimes \Gamma( V)$ (see Proposition~\ref{prop:vdefn}).
Their direct sum $\mathfrak h\oplus \mathfrak v$ admits a~minimally extended $L_\infty$ algebra structure from $\mathfrak h$. A~geometric deformation of this structure refers to perturbing the $L_\infty$ structure on $\mathfrak h\oplus \mathfrak v$ in a~way that preserves the bundle data of $V$ (Definition~\ref{geomdefdef}). Such deformations correspond exactly to coextensions of $\mathfrak h$ by~$\mathfrak v$ (see Proposition~\ref{Proposition2.8}). We prove (see Theorem~\ref{Theorem~1}) that the geometric deformations correspond specifically to Lie algebroid\footnote{In this paper, the term \emph{Lie algebroid} refers to the standard definition: a~vector bundle $V \to M$ equipped with a~bundle map $\rho\colon V \to TM$ (the anchor) and a~Lie bracket on $\Gamma(V)$ satisfying the Leibniz rule.} structures on $V$, and clarify in what sense two $L_\infty$ structures on the same underlying space are homotopy equivalent. The result is obtained through a~recursive computation generalizing Fedosov's method~\cite{CFT, Dolgushev,Fedosov}. In~\cite{GG_l}, it was shown that there exists a~functor embedding Lie algebroids into the category of $L_\infty$ algebra. However, the precise image of such embedding was not given. By the insightful observation of Vaintrob~\cite{Vain}, there is a~bijection between Lie algebroid structures on the vector bundle $V\to M$ and the homological vector fields on the $\mathbb Z$-graded manifold $V[1]$. Then our findings can be viewed as an $L_\infty$-algebraic extension of this perspective, suggesting that forming Lie algebroids is the only viable approach to encoding additional multiplicative structures on a~vector bundle. Furthermore, by the result of~\cite{OZ}, the geometric deformation necessarily induces an $L_\infty$ action of~$V$ on~$T_M$.

Then we provide concrete characterizations of the Atiyah and Atiyah--Chern classes of the geometrically deformed $L_\infty$ algebra through local computations.
Since the $l_n$ operations are given by specific tensors in geometric terms, it is possible to describe a~few lower operations concretely. The covariant computations yield tensor representatives of the Atiyah class (see Proposition~\ref{atiyah0part}) and the Atiyah--Chern classes.
The Atiyah--Chern classes reside within the de Rham complex of $\mathfrak h\oplus \mathfrak v$, which is weakly equivalent to the Weil algebra\footnote{This was first shown by~\cite{GG_l} using homotopical methods. Our approach uses a~local computation with an explicit homotopy retraction.} of $V$ through a~homotopy retraction (see Proposition~\ref{propweilclass}). Despite being trivial at cohomology level, the Atiyah--Chern classes remain interesting quantities to look into, which may shed light on the equivariant character theory of Lie algebroids. We show that a~representative of the first Atiyah--Chern class is the ${\rm d}^{\rm DR}$-image of a~cocycle representing the modular class~\cite{ELW} (see Proposition~\ref{modularc1}). As for the general Atiyah--Chern classes originating from an action Lie algebroid, they give equivariant Chern characters in the leading order (see Proposition~\ref{equivarKalk}).

Finally, we discuss some applications in Batalin--Vilkovisky (BV) theories. The Alexandrov--Kontsevich--Schwarz--Zaboronsky (AKSZ) method has been used to describe classical aspects of these models, whose geometric input can be encoded in an $L_\infty$ algebra together with an admissible pairing~\cite{AKSZ, Co11}.
We show that the geometric deformation of the $L_\infty$ structures in topological cotangent theories that preserves their admissible pairings leads in a~unique way to the Poisson sigma models (see Theorem~\ref{posthm}).
From the perspective of model building, while~\cite[Chapter 12]{CG2} demonstrates how $L_\infty$ algebra extensions model classical BV theories with symmetries, coextensions given by the geometric deformation yield a~broader class of theories, among which are the cotangent theories as described in Proposition~\ref{prop:cotang}.
The one-loop partition function of those cotangent theories combines the Atiyah--Chern classes of the target $L_\infty$ algebra with arithmetic invariants of the worldsheet geometry~\cite{Co11, GLL, GG_ahat}.
Therefore, the Atiyah--Chern classes computed in Section~\ref{atiyah} provide critical insights into the one-loop behavior of topological cotangent theories (see Proposition~\ref{prop:cotang}).

There are several relevant results to our work that we'd like to highlight.
It is crucial that the $L_\infty$ structures explored in the current paper are dependent on the choice of connections for both the tangent bundle and the Lie algebroid. Similar settings have been studied and led to the theory of secondary characters~\cite{Crainic00, CF,Fernandes}, see also~\cite{Blaom, KS} for the discussion on the compatible connections. Despite the fact that the secondary characters reside in Lie algebroid cohomology rather than the Weil algebra~\cite{AAC11,AAC}, the latter is a~natural location for Atiyah--Chern classes, which hopefully brings us new information on the equivariant theory.

The Atiyah class of Lie algebroids has been studied extensively in various settings, including smooth algebraic varieties~\cite{CVB}, dg-manifolds~\cite{MSX} and Lie pairs~\cite{CSX}. In both the dg-manifold and Lie pair settings, it has been shown that the Atiyah and Atiyah--Chern classes coincide for regular Lie algebroids~\cite{CXX, Xiang}. These results can be compared with the construction using $L_\infty$ spaces~\cite{GG_l}. Our computation serves as a~complement to the above results and is particularly applicable to non-regular Lie algebroids.
 A~paper by Calaque, Campos, and Nuiten~\cite{CCN} establishes an equivalence of $\infty$-categories between $L_\infty$-algebroids over a~smooth manifold $M$ and certain curved $L_\infty$ algebras arising from graded bundles.
While the functor from the $L_\infty$-algebroid side to the $L_\infty$ algebra side is made fairly explicit (see~\cite[Sections~5.2 and~6.1]{CCN}), the inverse direction is described only abstractly, without a~concrete construction on representatives.
The present work provides such an explicit construction for a~distinguished subclass of objects --- namely, $L_\infty$ algebras that contain a~fixed $L_\infty$ subalgebra encoding the smooth structure of $M$.
Although our notion of geometric deformation does not encompass all fibrant-cofibrant objects in the CCN model, it defines a~particularly interesting and computable subclass.
It would be desirable to describe, for general graded bundles, a~geometric fibrant replacement adapted to this setting, paralleling the construction of fibrant replacements in~\cite[Section~6.1]{CCN}. A~clearer geometric understanding of such fibrant replacements should also strengthen the study of geometric deformations and of the corresponding Atiyah classes of $L_\infty$ algebras over dg manifolds~\cite{SSX}, especially once the computation of Atiyah--Chern characters is extended to the graded-bundle case.\looseness=-1

The paper is organized as follows.

In Section~\ref{linfty}, we introduce the geometric deformation on the semidirect sum $L_\infty$ algebra~${\mathfrak h\oplus \mathfrak v}$. We use an adjusted version of Fedosov's method to demonstrate that the requirement for the geometric deformation can be solved recursively.
Then we compute lower brackets of the deformed $L_\infty$ algebra, and show that the nilpotency of the differential results in precise conditions for a~vector bundle to become a~Lie algebroid. We also investigate the homotopy equivalences of different deformed structures.

In Section~\ref{atiyah}, we define and calculate the Atiyah and Atiyah--Chern classes of $L_\infty$ algebras, which are cohomological classes that encode both the smooth geometry of bundles and the nonlinear products of Lie algebroids. We discuss the connection between Atiyah--Chern classes and existing results.

In Section~\ref{aksz}, we give a~brief review of cotangent BV field theories that take $L_\infty$ algebras as input. Then we show that any geometric deformation, applied to a~two-dimensional cotangent theory, necessarily produces a~Poisson sigma model.

In Appendix~\ref{appA}, we provide a~comprehensive review of the undeformed curved $L_\infty$ algebra which encodes the smooth geometry. We modify Fedosov's method that facilitates recursive computation of the $L_\infty$ structure while allowing for arbitrary gauge choices. The homotopy equivalences of pertinent structures are discussed.

In Appendix~\ref{appB}, some lower brackets of the geometrically deformed $L_\infty$ algebra are computed.

\section[Curved L\_infty structures from vector bundles and geometric deformation]{Curved $\boldsymbol{L_\infty}$ structures from vector bundles\\ and geometric deformation}
\label{linfty}

In this section, we begin by reviewing the construction of the curved $L_\infty$ algebra associated with a~real vector bundle, and subsequently explore deformations of the geometric type.

\subsection{General framework}

Throughout this paper, the base field is assumed to be $\mathbb R$. Let $\mathbb V$ be a~$\mathbb Z$-graded vector space. For any $n$-tuple of elements $v_1, \dots, v_{n}\in \mathbb V$ and $\tau \in S_n$, the {\it Koszul sign factor} is a~$\{1,-1\}$-valued function defined by the equation
\[v_{\tau(1)}\cdot v_{\tau(2)}\cdots v_{\tau(n)} = \epsilon_{\tau}^{v_1,\dots, v_{n}} v_1\cdot v_2\cdots v_{n}\]
in $\operatorname{Sym}^{n} \mathbb V$. The Koszul sign factor naturally extends to modules defined over unital graded commutative algebras.
Let $\mathrm{Sh}(p,q)$ be the set of $(p,q)$ unshuffles inside the permutation group~$S_{p+q}$ in $p+q$ letters.

Within this paper, we assume that all the commutative dg algebras are non-negatively graded, unital, and equipped with a~nilpotent differential ideal.

\begin{Definition}\label{linftydefn}
Let $(A, {\rm d}_A)$ be a~commutative dg algebra (so, in particular, it has a~nilpotent differential ideal $I_A$). A~curved $L_\infty$ algebra over $(A,{\rm d}_A)$ is a~$\mathbb Z$-graded finitely generated projective $A$-module $\mathbb V$, equipped with a~collection of $(A,{\rm d}_A)$-linear maps (i.e., the $n$-ary brackets)
\[l_n\colon \ \operatorname{Sym}^{n}_{A} (\mathbb V[1])\to \mathbb V[1]\]
of degree $1$ for $n\geq0$, such that
\begin{enumerate}
\itemsep=0pt
\item[$(1)$] for every $k$-tuple $v_1, \dots, v_k$ of homogeneous elements in $\mathbb V[1]$,
\[\sum_{i=0}^k\sum_{\tau\in \mathrm{Sh}(i,k-i)} \epsilon_{\tau}^{v_1,\dots, v_k} l_{k-i+1}(l_{i}(v_{\tau(1)},\dots, v_{\tau(i)}),v_{\tau(i+1)},\dots, v_{\tau(k)})=0,\]
\item[$(2)$] the $0$-ary bracket $l_0\colon A\to \mathbb V[1]$ maps $1\in A$ to an element in the submodule $I_A(\mathbb V[1])$.
\end{enumerate}
\end{Definition}

The $(A,{\rm d}_A)$-linearity condition requires that all maps $\{l_n\}_{n=0,n\neq1}^\infty$ be $A$-linear, whereas $l_1$ satisfies the Leibniz rule
\[l_1(fv) =( {\rm d}_Af )v+(-1)^{|f|} f l_1\,(v),\qquad\forall f\in A,\quad v\in \mathbb V.\]

\begin{Remark} A~curved $L_\infty$ algebra structure defined on $\mathbb V$ as above corresponds to a~(curved) $L_\infty[1]$ algebra structure on $\mathbb V[1]$ in literature such as~\cite{SSX}. In both cases, the Koszul sign factors are calculated according to the degree in $\mathbb V[1]$ rather than $\mathbb V$.
\end{Remark}

The conditions on the $\{l_n\}$ operations can be encoded\footnote{By virtue of the f.g. projective module condition, $\mathbb V$ becomes reflexive, and as such, taking dual is reversible.} into a~coboundary condition ${\rm d}^2=0$ on the Chevalley--Eilenberg (CE) cochain
\[C^*(\mathbb V):= \widehat{\operatorname{Sym}}_A\bigl(\mathbb V^\vee[-1]\bigr),\]
where the differential is a~degree $1$ derivation (on the $(A,{\rm d}_A)$-algebra) fully determined by the duals of all $l_n$ operations,
\[d\omega = \sum_{n=0}^\infty \frac1{n!}l_n^\vee (\omega), \qquad \forall \omega\in \mathbb V^\vee[-1].\]
The final condition in Definition~\ref{linftydefn} involving the nilpotent ideal $I_A$ now takes the form of requiring that $d$ preserves the ideal \smash{$\widehat{\operatorname{Sym}}^{\geq1}_A\bigl(\mathbb V^\vee[-1]\bigr)$} up to elements in $I_A$.

Within the context, we introduce the notion of morphisms between two $L_\infty$ algebras and homotopy equivalences.

\begin{Definition}
Given two $L_\infty$ algebras $\mathfrak g_1$ and $\mathfrak g_2$ over a~commutative dg algebra $(A, {\rm d}_A)$, an~$L_\infty$ morphism from $\mathfrak g_1$ to $\mathfrak g_2$ is a~commutative dg algebra map
\[C^*(\mathfrak g_2)\to C^*(\mathfrak g_1)\]
that preserves the filtration defined by the nilpotent ideal $I_A\subset A$. An $L_\infty$ morphism is called {\it strict} if the corresponding commutative dg algebra map is linear, i.e., the map preserves the degree of symmetric powers on both sides.
\end{Definition}

\begin{Definition} A~flat family of curved $L_\infty$ algebra structures on $\mathbb V$ parameterized by the standard $n$-simplex $\Delta^n$ consists of a~curved $L_\infty$ algebra over $\bigl(\Omega_{\Delta^n}\otimes_{\mathbb R} A$, ${\rm d}_t\otimes 1+1\otimes {\rm d}_A\bigr)$, i.e., a~degree $+1$ derivation \smash{${\rm d}^{(t)}$} on the extended CE complex $\Omega_{\Delta^n}\otimes_{\mathbb R} C^*(\mathbb V)$, together with a~splitting~\smash{${\rm d}^{(t)}= {\rm d}_{\rm CE}^{(t)}+D_t$} where
\begin{enumerate}\itemsep=0pt
\item[(1)] for every $t\in \Delta^n$, the restriction of ${\rm d}_{\rm CE}^{(t)}$ to $C^*(\mathbb V)$ is the CE differential of a~curved $L_\infty$ algebra structure on $\mathbb V$,
\item[(2)] the operator $D_t$ is a~connection on the (trivial) $C^*(\mathbb V)$-bundle over $\Delta^n$ satisfying $\bigl[{\rm d}_{\rm CE}^{(t)}, D_t\bigr]\allowbreak+\frac12[D_t, D_t]=0$,
\item[(3)] the total differential ${\rm d}^{(t)}$ preserves the ideal \smash{$\Omega_{\Delta^n}\otimes_{\mathbb R} \widehat{\operatorname{Sym}}^{\geq1}_A\bigl(\mathbb V^\vee[-1]\bigr)$} up to $\Omega_{\Delta^n}\otimes_{\mathbb R}I_A$.
\end{enumerate}
Two curved $L_\infty$ algebra structures ${\rm d}_{\rm CE}^0$ and ${\rm d}_{\rm CE}^1$ are $n$-homotopic if they appear as the restrictions to the $0$-th and $n$-th vertices, respectively, of a~flat $\Delta^n$-family of curved $L_\infty$ algebras.
\end{Definition}

\begin{Definition}
Let $\mathfrak g$ be a~curved $L_\infty$ algebra over the unital commutative dg algebra $(A,{\rm d}_A)$. An $L_\infty$ module of $\mathfrak g$ is a~finitely generated $\mathbb Z$-graded projective $A$-module $\mathbb V$, equipped with a~collection of $(A,{\rm d}_A)$-linear maps (i.e., the $n$-ary brackets)
\[l'_n\colon\ \operatorname{Sym}^{n-1}_{A} (\mathfrak g[1])\otimes_A \mathbb V[1] \to \mathbb V[1]\]
of degree $1$ for $n\geq1$, such that
for any $k$-tuple $u_1, \dots, u_k$ of homogeneous elements in $\mathfrak g[1]$ and for any element $v\in \mathbb V[1]$
\begin{gather*}
\sum_{i=0}^k\sum_{\tau\in \mathrm{Sh}(i,k-i)} \epsilon_{\tau}^{u_1,\dots, u_k} l'_{k-i+1}(l_{i}(u_{\tau(1)},\dots, u_{\tau(i)}),u_{\tau(i+1)},\dots, u_{\tau(k)}, v)\\
\qquad{}+\sum_{i=1}^k\sum_{\tau\in \mathrm{Sh}(k-i,i)} \epsilon_{\tau}^{u_1,\dots, u_k}
(-1)^{(\sum_{j=1}^{k-i} |u_{\tau(j)}|)(1+\sum_{j=k-i+1}^{k} |u_{\tau(j)}|+|v|)}\\
\hphantom{\qquad{}+}{}\times
l'_{k-i+1}(u_{\tau(1)},\dots, u_{\tau(k-i)},l'_{i}(u_{\tau(k-i+1)},\dots, u_{\tau(k)},v))=0.
\end{gather*}
\end{Definition}

\begin{Definition}
Fix a~(curved) $L_\infty$ algebra $\mathfrak g$. Two $\mathfrak g$-module structures on $\mathbb V$ are homotopic, if they induce homotopic $L_\infty$ algebra structures on $\mathfrak g\oplus \mathbb V$.
\end{Definition}

For any $L_\infty$ $\mathfrak g$-module $\mathbb V$, there exists a~natural $L_\infty$ algebra structure on $\mathfrak g\oplus \mathbb V$, where the $n$-ary brackets are uniquely determined by $\{l_n^\vee \}_{n=0}^\infty$ and $\{(l'_n)^\vee\}_{n=1}^\infty$. Moreover, the natural maps~${\mathfrak g\oplus \mathbb V\to \mathfrak g}$ and $\mathfrak g\to \mathfrak g\oplus \mathbb V$ are strict morphisms of $L_\infty$ algebras over $(A,{\rm d}_A)$.
This is referred to as a~{\it split} Abelian extension (resp.\ the semidirect sum) of $L_\infty$ algebras $\mathfrak g$ by (resp.\ and) $\mathbb V$.

It is possible to deform the $L_\infty$ structure on $\mathfrak g\oplus \mathbb V$ while preserving only the strict $L_\infty$ morphism $\mathfrak g\to \mathfrak g\oplus \mathbb V$. Then the CE differential on the deformed $L_\infty$ algebra must necessarily preserve the ideal generated by $\mathbb V^\vee[-1]$ in $C^*(\mathfrak g\oplus \mathbb V)$. We also want the $L_\infty$ module structure on $\mathbb V$ to be easily distinguished from any new structures, which leads to the following definition of geometric deformation.

\begin{Definition}\label{geomdefdef}
Let $\mathfrak g$ be a~curved $L_\infty$ algebra over $(A, {\rm d}_A)$ and let $\mathbb V$ be a~$\mathfrak g$-module.
Assigning weight $-1$ to $\mathbb V$, we say that an $L_\infty$-structure $\bigl\{\tilde l_n\bigr\}_{n\geq0}$ on $\mathfrak g\oplus \mathbb V$ {\it geometrically} deforms the split Abelian extension if
\begin{enumerate}
\itemsep=0pt
\item[(1)] each of the brackets $\tilde l_n$ has non-negative weights, and
\item[(2)] the weight-$0$-components of $\tilde l_n$,
$ \operatorname{Sym}^n (\mathfrak g[1])\to \mathfrak g[1]$ and
$\mathbb V[1] \otimes \operatorname{Sym}^n (\mathfrak g[1])\to \mathbb V[1]$
coincide with corresponding operations in the split Abelian extension.
\end{enumerate}
\end{Definition}

The definition of geometric deformation depends explicitly on the chosen $L_\infty$-module structure of $\mathbb V$. Indeed, this dependence is not merely a~requirement encompassing vector bundle information; rather, it also can be derived from a~dual extension problem of the $L_\infty$ algebra $\mathfrak g$.

\begin{Proposition}\label{Proposition2.8}
Let $\mathfrak g$ be a~curved $L_\infty$ algebra over $(A, {\rm d}_A)$, and let $\mathbb V$ be a~finitely generated $\mathbb Z$-graded projective $A$-module. Suppose that operations $\bigl\{\tilde l_n\bigr\}_{n=0}^\infty$ define an $(A, {\rm d}_A)$-linear curved~$L_\infty$ structure on $\mathfrak g\oplus \mathbb V$ such that
\begin{enumerate}
\itemsep=0pt
\item[$(1)$] the obvious inclusion
$\iota\colon \mathfrak g\to \mathfrak g\oplus \mathbb V$
is a~strict $L_\infty$ morphism,
\item[$(2)$] the curving is preserved by $\iota$, in the sense that
$\tilde l_0 = \iota\circ l_0$.
\end{enumerate}
Then $\mathbb V$ is an $L_\infty$-module of $\mathfrak g$ such that $\bigl\{\tilde l_n\bigr\}_{n=0}^\infty$ is a~geometric deformation in the sense of Definition~{\rm\ref{geomdefdef}}.
\end{Proposition}
\begin{proof}
Assigning weight $-1$ to $\mathbb V$ as before, the operations $\bigl\{\tilde l_n\bigr\}_{n=0}^\infty$ decompose accordingly. The weight $0$-components have to satisfy the quadratic $L_\infty$ relation since negatively weighted operations are not allowed, due to the property of $\iota$. Furthermore, the component \[ \operatorname{Sym}^n (\mathfrak g[1])\to \mathfrak g[1]\] for any $n$
is entirely fixed by the $L_\infty$ structure on $\mathfrak g$. So the remaining components of weight $0$-operations are forced to encode an $\mathfrak g$-module structure on $\mathbb V$.\end{proof}

The coextension problems are typically more complex than extension problems~\cite{AM}. The central result of this section is that in a~geometric context to be presented, we can provide a~full characterization of all such extensions.

\subsection{Geometric realization from vector bundles and the deformation}

We now specialize the general framework of the previous subsection to a~concrete geometric setting. The ensuing discussion will be centered on a~specific class of $L_\infty$ algebras constructed from vector bundles, within which we study the geometric deformations.

\begin{Proposition}[\cite{GLL,GG_ahat}]\label{prop:hdefn}\samepage
Given a~smooth manifold $M$ with a~Levi-Civita connection $\nabla$ on the tangent bundle $T_M$, there exists a~curved $(\Omega_M, d)$-linear $L_\infty$ algebra on the de Rham complex\footnote{Unless explicitly noted, the tensor product $\otimes$ is over the sheaf/ring of smooth functions $C^\infty_M$.}
\[\mathfrak h\equiv\Omega_{M}\otimes \Gamma( T_M )[-1]\]
 such that
 \begin{enumerate}
\itemsep=0pt
 \item[$(1)$] there exists an element $\mathbf c\in \Omega_{M}^{\geq1}\otimes \Gamma( T_M ) \subset\mathfrak h[1] $ that uniquely determines the $0$-ary operation through
 \[l_0\colon\ \Omega_M\equiv \operatorname{Sym}^0 \mathfrak h[1] \to \mathfrak h[1],\qquad f\mapsto f\wedge \mathbf c\, \qquad \forall f\in \Omega_M,\]
\item[$(2)$] the unary operation is determined by the $\Omega_M$-extended connection $\nabla$ via
\[l_1\colon\ \mathfrak h[1]\equiv \operatorname{Sym}^1 \mathfrak h[1] \to \mathfrak h[1],\qquad fY\mapsto dfY+(-1)^{|f|} f\nabla Y\]
for all $f\in \Omega_M$ and $Y\in \Gamma( T_M )$,
 \item[$(3)$] up to a~gauge choice,\footnote{The gauge choice can be understood as some further conditions needed to uniquely specify $\{l_n\}_{n=2}^\infty$ once $l_0$ and $l_1$ are chosen. Within the context of the current proposition, different gauge choices lead to homotopic $L_\infty$ structure, see Appendix~\ref{appA}.} further $n$-ary brackets are uniquely defined by the $L_\infty$ relations and the previous two brackets.
 \end{enumerate}
 Furthermore, different choices of connections on $T_M$ lead to homotopic $L_\infty$ structures.
\end{Proposition}

\pagebreak

By fixing a~set of local coordinates \smash{$\bigl\{x^i\bigr\}_{i=1}^{\dim (M)}$} on manifold $M$, we uniquely determine a~frame ${\bigl\{y_i\equiv \frac\partial {x^i}\bigr\}_{i=1}^{\dim (M)}}$ on $T_M$. The curving $\mathbf c$ is given by ${\rm d}x^i\otimes y_i\in \mathfrak h[1]$.
 Under the coordinate change,
 \begin{gather*}
 \bigl\{x^i\bigr\}_{i=1}^{\dim (M)} \to \bigl\{\tilde{x}^i\bigr\}_{i=1}^{\dim (M)},\qquad
 {\rm d}\tilde x^i=\frac{\partial \tilde x^i}{\partial x^j} {\rm d}x^j, \qquad \tilde y_i= \frac{\partial x^j}{\partial \tilde x^i} y_j,
 \\
 {\rm d}\tilde x^i\otimes \tilde y_i = \frac{\partial \tilde x^i}{\partial x^j} {\rm d}x^j\otimes \frac{\partial x^{k}}{\partial \tilde x^i} y_{k}= \frac{\partial x^{k}}{\partial \tilde x^i}\frac{\partial \tilde x^i}{\partial x^j}{\rm d}x^j\otimes y_{k}.
 \end{gather*}
 This verifies that $\mathbf c$ is a~well-defined object in $\Omega^1_M\otimes\Gamma(T_M)$.

Higher operations can be computed recursively from the CE cochain \smash{$\widehat{\operatorname{Sym}} \bigl(\mathfrak h^\vee[-1]\bigr)$}. We demonstrate this fact in Lemma~\ref{recursivelemma} and Proposition~\ref{propsolvhstru}, which are presented in Appendix~\ref{appA}. Furthermore, we establish in Proposition~\ref{htpyeqnhstru} that these structures are homotopically equivalent.

\begin{Remark} A~gauge choice can be made by specifying the desired values of the dual brackets~$\{l_{n}^\vee\}_{n=2}^\infty$ under the homotopy contraction map $\delta^{-1}$ on the CE cochain $\widehat{\operatorname{Sym}} \bigl(\mathfrak h^\vee[-1]\bigr)$, which locally is represented as
\[\delta^{-1}=\begin{cases}
\displaystyle \frac1{p+q} y^i\iota_{\frac\partial{\partial x^i}}, &  \mathrm{on}\, \Omega_M^p\otimes \Gamma\bigl(\operatorname{Sym}^q\bigl( T_M^\vee\bigr)\bigr), \qquad p+q>0,\\
0, &  \mathrm{on}\, \Omega_M^0\otimes \Gamma\bigl(\operatorname{Sym}^0\bigl( T_M^\vee\bigr)\bigr).
\end{cases}\]
In the case of~\cite{Fedosov} and implicitly in the jet-bundle approach~\cite{GG_ahat}, a~popular choice is to require that all these images vanish. This particular choice is referred to as the $\delta^{-1}$ gauge.
\end{Remark}

\begin{Proposition}[\cite{GG_ahat}] \label{prop:vdefn}
Given the $L_\infty$ algebra $\mathfrak h$ as defined in Proposition~{\rm\ref{prop:hdefn}}, and given a~vector bundle $V$ of finite rank over $M$ with a~connection $\nabla'$, there exists an $L_\infty$ $\mathfrak h$-module structure, over the commutative dg algebra $(\Omega_M,{\rm d})$, on the de Rham complex
$\mathfrak v\equiv\Omega_{M}\otimes \Gamma( V )$
 such that
 \begin{enumerate}
\itemsep=0pt
\item[$(1)$] the unary operation is determined by the $\Omega_M$-extended connection $\nabla'$ via
\[l'_1\colon\ \mathfrak v[1] \to \mathfrak v[1],\qquad f\alpha \mapsto {\rm d}f\alpha+(-1)^{|f|}f\nabla' \alpha\]
for all $ f\in \Omega_M$ and $\alpha\in \Gamma(V)$,
 \item[$(2)$] up to a~gauge choice, further $n$-ary brackets $\{l'_n\}_{n=2}^\infty$ are uniquely defined by the $L_\infty$ relations and the previous unary bracket.
 \end{enumerate}
 Furthermore, different choices of connections on $V$ lead to homotopic $L_\infty$ module structures.
\end{Proposition}
The proof is similar to that of Proposition~\ref{prop:hdefn}.

The following theorem highlights an intimate connection between the geometric deformations of $\mathfrak h\oplus \mathfrak v$ and the Lie algebroid structures on $V$.
\begin{Theorem}\label{Theorem~1}
Given a~curved $L_\infty$ algebra $\mathfrak h$ and its module $\mathfrak v$ as defined in Propositions~{\rm\ref{prop:hdefn}} and~{\rm\ref{prop:vdefn}}, any $(\Omega_{M},{\rm d})$-linear geometric deformation from the split Abelian extension $\mathfrak h\oplus\mathfrak v$
 is uniquely determined by a~Lie algebroid structure on $V$.
\end{Theorem}
The proof involves the following steps:
\begin{enumerate}
\itemsep=0pt
\item[$(1)$] Degree analysis identifies the nontrivial operations, which need to satisfy a~defining equation and a~consistency equation.
\item[$(2)$] Recursive solution of the defining equation can be obtained, which is specified uniquely by the ``initial conditions'' (see Proposition~\ref{propSol}).
\item[$(3)$] The consistency equation requires that the ``initial conditions'' follow from a~Lie algebroid structure on $V$ (see Proposition~\ref{prop:quadrel}).
\end{enumerate}

\subsection{Step 1: Degree analysis}

On the underlying space $\mathfrak h\oplus\mathfrak v$, we consider possible new operations that are allowed in a~geometric deformation. These operations are of the forms
\[\rho_{m+n}\colon\ \operatorname{Sym}^m (\mathfrak h[1])\otimes \operatorname{Sym}^n (\mathfrak v [1])\to \mathfrak h[1], \qquad n\geq1,\quad m\geq0,\]
and
\[\mu_{m+n}\colon\ \operatorname{Sym}^m (\mathfrak h[1])\otimes \operatorname{Sym}^n (\mathfrak v[1]) \to \mathfrak v[1], \qquad n\geq2,\quad m\geq0.\]

For any element $\alpha \in \mathfrak h[1]\oplus \mathfrak v[1]\cong \Omega_M\otimes\Gamma(T_M\oplus V[1])$, we say that $\alpha$ is of {\it form degree}~$k$, if~${\alpha\in \Omega^k_M\otimes\Gamma(T_M\oplus V[1])}$, and is of {\it internal degree} $0$ (resp.\ $-1$) if it takes value in $\Gamma(T_M)$ (resp.~$\Gamma(V[1])$). The total degree is the sum of form degree and internal degree. Analogous degree assignments apply to tensors and duals of $ \mathfrak h[1]\oplus \mathfrak v[1]$.

The operation $\rho_{m+n}$ has a~total degree of $1$ and an internal degree of $n$. Therefore, its form degree is $1-n\geq0$. Since $n\geq1$, we have $n=1$. Thus, the operation $\rho_{m+n}$ has a~form degree of~$0$ and corresponds to a~bundle map.\footnote{The conclusion that $\rho_1$ is a~bundle map may seemingly contradict the fact that in general the $l_1$ operation on an $L_\infty$ algebra $\mathfrak g$ over a~commutative dg algebra $(A,{\rm d}_A)$ need not be $A$-linear. In particular, we shall have~${l_1(fX) = \pm f l_1(X)+{\rm d}_Af X}$ for $f\in A$ and $X\in \mathfrak g$. But the operation $\rho_1$ has no nontrivial form degree, and hence there is no room for the second term to show up. Alternatively, we may regard $\rho$ as an ``off-diagonal'' component of a~connection on the graded bundle $T_M\oplus V[1]$.} Similarly, the operation $\mu_{m+n}$ has a~total degree of $1$ and an internal degree of $n-1$. So in total, only operations $\{\rho_{m+1}\}_{m=0}^\infty$ and $\{\mu_{m+2}\}_{m=0}^\infty$ survive, which are induced from the corresponding bundle maps.

To solve the constraints on the $L_\infty$ operations, we consider the CE cochain complex
\[C^*(\mathfrak h\oplus \mathfrak v)\equiv\widehat{\operatorname{Sym}}\bigl(\mathfrak h^\vee[-1]\oplus \mathfrak v^\vee [-1]\bigr).\]
We let ${\rm d}^{\nabla}$ denote the CE differential on \smash{$\widehat{\operatorname{Sym}}\bigl(\mathfrak h^\vee[-1]\bigr)$}, and let ${\rm d}^{\nabla'}$ denote the CE differential dualizing the $L_\infty$ module structure on $V$.
Denote by \smash{$\{y_i\}_{i =1}^{\dim  (M)}$} and \smash{$\{v_a\}_{a =1}^{\mathrm{rk} (V)}$} local frames for bundle $T_M$ and $V$, respectively, \big(likewise \smash{$\bigl\{y^i\bigr\}_{i =1}^{\dim  (M)}$} and \smash{$\{v^a\}_{a =1}^{\mathrm{rk} (V)}$} for the dual frames\big). Locally we have that
\begin{gather*}
{\rm d}^{\nabla} = \delta + \nabla+
\sum_{n>1} \frac1{n!}\bigl( (l_{n} )^i_{i_1,i_2,\dots, i_n}\, y^{i_1} \cdots y^{i_n} \wedge -\bigr) \circ \frac\partial{\partial y^i}
, \qquad \delta={\rm d}x^i\frac\partial{\partial y^i}.\\
{\rm d}^{\nabla'} = \nabla'+
\sum_{n>0} \frac1{n!} \bigl( (l'_{1+n} )^b_{i_1,i_2,\dots, i_n, a}\, y^{i_1} \cdots y^{i_n} v^a \wedge -\bigr) \circ \frac\partial{\partial v^b}.
\end{gather*}
The coefficient $(l_{n} )^i_{i_1,i_2,\dots, i_n}$ is given by the $y_i$-component of $l_n(y_{i_1}, y_{i_2}, \dots, y_{i_n})$, and similarly for other coefficients.
The connections $\nabla$ and $\nabla'$ on $T_M$ and $V$ induce connections on the products of dual bundles $\widehat{\operatorname{Sym}} \bigl(T^\vee_M\bigr)$ and $\widehat{\operatorname{Sym}}\bigl(V^\vee[-1]\bigr)$, which are again denoted by $\nabla$ and $\nabla'$ unless explicitly stated otherwise.
The $L_\infty$ structure of the semi-direct product is prescribed by the quadratic relation $\bigl({\rm d}^{\nabla}+{\rm d}^{\nabla'}\bigr)^2=0$, and all the higher brackets can be solved recursively as shown in Appendix~\ref{appA}.
 A~geometrically deformed $L_\infty$ structure on $\mathfrak h\oplus \mathfrak v$ has new components in the CE differential, which consist of duals of brackets $\{\rho_{m+1}\}_{m=0}^\infty$ and $\{\mu_{m+2}\}_{m=0}^\infty$,
\begin{gather}
{\rm d}^\rho:= \sum_m \frac1{m!} \bigl( (\rho_{m+1} )^i_{i_1,i_2,\dots, i_m, a} y^{i_1} \cdots y^{i_m} v^a \wedge -\bigr) \circ \frac\partial{\partial y^i},\\
{\rm d}^\mu:= \sum_m \frac1{m!2!}\bigl( (\mu_{m+2} )^c_{i_1,i_2,\dots, i_m, a, b} y^{i_1} \cdots y^{i_m} v^a \wedge v^b\wedge -\bigr) \circ \frac\partial{\partial v^c}.
\end{gather}

To simplify the notations, we shall denote
\[\frac1{n!}\bigl( (l_{n} )^i_{i_1,\dots, i_{n}}\, y^{i_1} \cdots y^{i_{n}} \wedge -\bigr) \circ \frac\partial{\partial y^i}, \qquad n\geq2\] by ${\rm d}^{\nabla}_n$, and similarly for ${\rm d}^{\nabla'}_{n+1}$ $(n\geq1)$, ${\rm d}^{\rho}_{n+1}$ and ${\rm d}^\mu_{n+2}$ $(n\geq0)$.
To make notations further compatible, we shall also use \smash{${\rm d}^\nabla_1$} and \smash{${\rm d}^{\nabla'}_1$} to denote $\nabla$ and $\nabla'$, respectively.

The quadratic relation of the CE differential is subject to the following weight decomposition:
\begin{gather}
\bigl[{\rm d}^{\nabla}+{\rm d}^{\nabla'}, {\rm d}^\rho + {\rm d}^\mu\bigr] =0,
\label{eqnsol}\\
\bigl[{\rm d}^\rho + {\rm d}^\mu, {\rm d}^\rho + {\rm d}^\mu\bigr] =0.
\label{eqnLie}\end{gather}

Furthermore, in the $n$-th component $(n\geq 3)$, equation~\eqref{eqnsol}
decomposes into
\[\tag{\ref{eqnsol}.n} -\bigl[\delta, {\rm d}^\mu_{n}+{\rm d}_{n}^\rho\bigr]=\sum_{1\leq k\leq n-1}\bigl[{\rm d}_k^\nabla+{\rm d}^{\nabla'}_k, {\rm d}^\rho_{n-k}\bigr]+\sum_{1\leq k\leq n-2}\bigl[{\rm d}_k^\nabla+{\rm d}^{\nabla'}_k, {\rm d}^\mu_{n-k}\bigr].
\]
For consistency, we have
\begin{gather*}
\tag{\ref{eqnsol}.2} -\bigl[\delta, {\rm d}^\mu_{2}+{\rm d}_{2}^\rho\bigr]=\bigl[{\rm d}_1^\nabla+{\rm d}^{\nabla'}_1, {\rm d}^\rho_{1}\bigr],\\
\tag{\ref{eqnsol}.1} -\bigl[\delta, {\rm d}_{1}^\rho\bigr]=0.
\end{gather*}

We shall refer to equation~\eqref{eqnsol} as the defining equation, and to equation~\eqref{eqnLie} as the consistency equation.

\subsection{Step 2: Defining equation}

In this subsection, equation~\eqref{eqnsol} is solved recursively using a~version of Fedosov's method.
Fedosov's method was initially devised to locate a~specific ``Hamiltonian'' in the Weyl bundle, which corresponds to a~differential in symplectic case. When dealing with non-symplectic scenarios such as in our context, the task shifts to solving for derivation operators $\{{\rm d}^\rho_{n}\}_{n=1}^\infty$ and~$\{{\rm d}^\mu_{n}\}_{n=2}^\infty$ directly. While the original method successfully employs the homotopy contraction $\delta^{-1}$ on the Weyl bundle, it is important to note that the contraction $\delta^{-1}$ does not automatically induce a~natural contraction on the space of operators on $C^*(\mathfrak h\oplus \mathfrak v)$. Although a~precise homotopy contraction is known to exist for the de Rham complex of formal (poly-)vector fields~\cite{Dolgushev}, direct extensions of this method to the deformed components $\{{\rm d}^\rho_{n}\}_{n=1}^\infty$ and $\{{\rm d}^\mu_{n}\}_{n=2}^\infty$ encounter additional challenges involving some degree-related issues.

We adopt a~slightly different strategy. The idea involves solving a~class of elements $\bigl\{{\rm d}^\rho_{n}\bigl(y^i\bigr)\mid i=1,\dots,\dim (M)\bigr\}_{n=1}^\infty$ and $\bigl\{{\rm d}^\mu_{n}\bigl(v^a\bigr)\mid a=1,\dots, \mathrm{rk}(V)\bigr\}_{n=2}^\infty$ recursively. Through this process, we can acquire the essential information to determine the differential operators. The homotopy contraction $\delta^{-1}$ is still handy, and the gauge choices now become unique.\footnote{To appreciate the uniqueness, compare the current situation with Corollary~\ref{gaugechoice} in Appendix~\ref{appA}.} For the recursion, It is imperative to address the compatibility and solvability of this set of recursive equations as a~preliminary step.

For the set of equations~$\{(\ref{eqnsol}.n)\}_{n=1}^\infty$, the left-hand sides are $\delta$-exact.
We say that equation~$(\ref{eqnsol}.n)$ is {\it compatible} if its right-hand side commutes with $\delta$. Say that equation~$(\ref{eqnsol}.n)$ is {\it solvable} if its solution exists.

\begin{Lemma}\label{compatib}
If equation $(\ref{eqnsol}.k)$ is solvable for each $k\leq n$, then equation $(\ref{eqnsol}.n+1)$ is compatible.
\end{Lemma}

The proof is essentially the same as the compatibility result for equations of $\{{\rm d}^\nabla_k\}_{k=2}^\infty$ in Appendix~\ref{appA}, therefore it is omitted.

\begin{Proposition}\label{propSol}
The extended differential ${\rm d}^\mu+{\rm d}^\rho$ can be solved recursively and uniquely from equations
$\{(\ref{eqnsol}.n)\}_{n=1}^\infty$
for any fixed initial data
$\rho_1\colon \mathfrak v [1]\to \mathfrak h[1]$ and $\mu_2\colon \operatorname{Sym}^2 (\mathfrak v[1]) \to \mathfrak v[1]$.
\end{Proposition}

\begin{proof}
Locally, as the commutative dg algebra $C^*(\mathfrak g)$ is generated by $\bigl\{y^i, v^a\bigr\}$'s, to solve for~${\rm d}^\rho$ and ${\rm d}^\mu$, one needs to solve for ${\rm d}^\rho\bigl(y^i\bigr)$ and ${\rm d}^\mu\bigl(v^a\bigr)$, which are the coefficients in front of \smash{$\frac\partial{\partial y^i}$} and~\smash{$\frac\partial{\partial v^a}$}, respectively.

Taking into account the range of each component, the quadratic relation decomposes into
\begin{equation}
-\bigl[\delta,{\rm d}^\rho_{n+1}\bigr]\bigl(y^i\bigr)=\sum_{1\leq m\leq n} \bigl[{\rm d}^{\nabla}_m, {\rm d}^\rho_{n-m+1}\bigr]\bigl(y^i\bigr)+\sum_{1\leq m\leq n} {\rm d}^{\nabla'}_m\bigl({\rm d}^\rho_{n-m+1}\bigl(y^i\bigr)\bigr),
\label{soleleeqn1}
\end{equation}
likewise for ${\rm d}^\mu$.
 Since $\bigl[\delta,{\rm d}^\rho_{n+1}\bigr]\bigl(y^i\bigr) = \delta\bigl({\rm d}^\rho_{n+1}\bigl(y^i\bigr)\bigr)$, the above equations for
\[
\bigl\{{\rm d}^\rho_{n+1}\bigl(y^i\bigr)\bigr\}_{n\in \mathbb Z_{\geq0}}\subset \Omega^0_M\otimes \Gamma\bigl(\widehat{\operatorname{Sym}}\bigl( T^\vee_M\bigr)\otimes V^\vee[-1]\bigr)\]
can be solved by a~homotopy contraction given by $\delta^{-1}$.
It is easy to check that
$\delta\circ \delta^{-1}+ \delta^{-1}\circ\delta = {\rm id} $ on \smash{$\Omega^p\otimes \Gamma\bigl(\operatorname{Sym}^q\bigl( T_M^\vee\bigr)\otimes \widehat{\operatorname{Sym}}^\bullet \bigl( V^\vee[-1]\bigr)\bigr)$}
for $p+q>0$.\footnote{Due to the failure in the case $p=q=0$, such equations have no prediction on ${\rm d}^\rho_{0+1}$ and ${\rm d}^\mu_{0+2}$.}

Now we can apply an inductive argument. Suppose \smash{$\bigl\{{\rm d}^\rho_k\bigr\}_{k=1}^n$} are solutions to equations $\{(\ref{eqnsol}.k)\}_{k=1}^n$.
We define \smash{${\rm d}^\rho_{n+1}$} on the local frame \smash{$\bigl\{y^i\bigr\}_{i=1}^{\dim (M)}$} by
\begin{equation}
-{\rm d}^\rho_{n+1}\bigl(y^i\bigr) = \sum_{1\leq m\leq n} \delta^{-1}\bigl[{\rm d}^{\nabla}_m, {\rm d}^\rho_{n-m+1}\bigr]\bigl(y^i\bigr)+\sum_{1\leq m\leq n} \delta^{-1}{\rm d}^{\nabla'}_m\bigl({\rm d}^\rho_{n-m+1}\bigl(y^i\bigr)\bigr).
\label{soleleeqn2}
\end{equation}
Then ${\rm d}^{\rho}_{n+1}$ is a~solution to equation $(\ref{eqnsol}.n+1)$.

In order to check that, observe that ${\rm d}_m^{\nabla'}\bigl(y^i\bigr) = 0$,
\begin{gather*}
 \sum_{1\leq m\leq n} \delta \left(\bigl[{\rm d}_m^\nabla, {\rm d}^\rho_{n-m+1}\bigr]\bigl(y^i\bigr) \right)=\sum_{1\leq m\leq n} \bigl[\delta,\bigl[{\rm d}_m^\nabla, {\rm d}^\rho_{n-m+1}\bigr]\bigr]\bigl(y^i\bigr) + \bigl[{\rm d}_m^\nabla, {\rm d}^\rho_{n-m+1}\bigr]\bigl(\delta y^i\bigr),\\
\sum_{1\leq m\leq n} \delta \left({\rm d}_m^{\nabla'}\bigl({\rm d}^\rho_{n-m+1}\bigl(y^i\bigr) \bigr)\right)=\sum_{1\leq m\leq n} \bigl[\delta,\bigl[{\rm d}_m^{\nabla'}, {\rm d}^\rho_{n-m+1}\bigr]\bigr]\bigl(y^i\bigr) + \bigl[{\rm d}_m^{\nabla'}, {\rm d}^\rho_{n-m+1}\bigr]\bigl(\delta y^i\bigr).
\end{gather*}
All terms on the right-hand sides vanish by Lemma~\ref{compatib}.
So
\begin{align*}
& \delta\delta^{-1}\bigg(\sum_{1\leq m\leq n} \bigl[{\rm d}^{\nabla}_m, {\rm d}^\rho_{n-m+1}\bigr]\bigl(y^i\bigr)+ {\rm d}^{\nabla'}_m\bigl({\rm d}^\rho_{n-m+1}\bigl(y^i\bigr)\bigr)\bigg) \\
&\qquad =  \sum_{1\leq m\leq n} \bigl[{\rm d}^{\nabla}_m, {\rm d}^\rho_{n-m+1}\bigr]\bigl(y^i\bigr)
+ {\rm d}^{\nabla'}_m\bigl({\rm d}^\rho_{n-m+1}\bigl(y^i\bigr)\bigr),
\end{align*}
which is $-\delta\left({\rm d}^\rho_{n+1}\bigl(y^i\bigr) \right)\equiv -\bigl[\delta, {\rm d}^\rho_{n+1}\bigr]\bigl(y^i\bigr)$.

The induction starts with an arbitrary choice for ${\rm d}^\rho_1$, which is the linear dual of $\rho_1$ and satisfies equation $(\ref{eqnsol}.1)$ trivially.

Fixing the operator ${\rm d}_1^\rho$, the uniqueness of recursion that generates the entire differential ${\rm d}^\rho$ is established based on the uniqueness of ${\rm d}^\rho_{n}\bigl(y^i\bigr)$ for each $i$ and $n\geq2$. The latter can be attributed to the property that $\delta^{-1}$ and $\delta$ act as inverse operations between \smash{$ \Omega^0\otimes \Gamma\bigl( \widehat{\operatorname{Sym}}^{\geq1}\bigl( T_M^\vee\bigr)\otimes V^\vee[-1]\bigr)$} and \smash{$ \Omega^1\otimes \Gamma\bigl( \widehat{\operatorname{Sym}}^{\geq0}\bigl( T_M^\vee\bigr)\otimes V^\vee[-1]\bigr)$}. Hence the solution given by~\eqref{soleleeqn2} is ``isomorphically'' transformed from equation~\eqref{soleleeqn1}.

Now for solutions for ${\rm d}^\rho$, a~recursion can be applied to solve for ${\rm d}^\mu$, which involves choosing a~local basis $\{v^a\}$ for the bundle $V^\vee$ and solving
\begin{align*}
\begin{aligned}
-\delta\bigl({\rm d}^\mu_{n+2}\bigl(v^a\bigr)\bigr)={}&\sum_{1\leq m\leq n} \bigl[{\rm d}^{\nabla'}_m, {\rm d}^\mu_{n-m+2}\bigr]\bigl(v^a\bigr)+{\rm d}^{\rho}_{m}\bigl({\rm d}^{\nabla'}_{n-m+2}\bigl(v^a\bigr)\bigr) \\
&+ \sum_{1\leq m< n} {\rm d}^{\nabla}_m\bigl({\rm d}^\mu_{n-m+2}\bigr)\bigl(v^a\bigr)
\end{aligned}
\end{align*}
for $n\geq1$. The initial condition
$-\bigl[\delta, {\rm d}^\mu_2\bigr] = 0$
is trivially satisfied for any choice of ${\rm d}^\mu_2$, which is the linear dual of $\mu_2$. The details are similar to the previous steps and hence are omitted.
\end{proof}

\subsection{Step 3: Consistency equation}
\label{algoid}

Equation~\eqref{eqnsol} alone has been solved recursively and uniquely with arbitrary choices of ${\rm d}_1^\rho$ and~${\rm d}_2^\mu$ as initial data.
Further constraints may come from equation~\eqref{eqnLie}. To obtain a~geometric implication of those constraints, we shall first relate ${\rm d}_1^\rho$ and ${\rm d}_2^\mu$ to explicit bundle operations.
\begin{Proposition}
Given a~geometrically deformed $L_\infty$ algebra in $\delta^{-1}$ gauge,\footnote{The $\delta^{-1}$ gauge fixes the $L_\infty$ structure on $\mathfrak h$ and the module structure on $\mathfrak v$ entirely by picking up connections on $T_M$ and $V$, respectively, see Corollary~\ref{gaugechoice} in Appendix~\ref{appA}.}
if the initial data~${\rho\equiv\rho_1}$ and $\mu\equiv\mu_2$ are specified, then
\begin{gather*}
\rho_{1+1}(X,\alpha) =\nabla_X \rho(\alpha) -\rho(\nabla'_X \alpha),\\
\mu_{1+2}(X, \alpha,\beta) = -\bigl({\nabla'}_{X}^{(1)}\mu\bigr)(\alpha,\beta)+\frac12R_{\nabla'}(X,\rho(\alpha))(\beta) - \frac12R_{\nabla'}(X,\rho(\beta))(\alpha),
\end{gather*}
for all $X\in\Gamma(T_M)$ and $\alpha\in\Gamma(V)$,
where
\[\bigl({\nabla'}_{X}^{(1)}\mu\bigr)(\alpha,\beta) = \nabla'_X\mu(\alpha,\beta) -\mu(\nabla'_X\alpha,\beta)-\mu(\alpha,\nabla'_X\beta).\]
\end{Proposition}

We now consider quadratic constraints imposed by equation~\eqref{eqnLie}.

\begin{Proposition}\label{prop:quadrel}
Let $\bigl(\sum_{n=0}^\infty {\rm d}^\rho_{n+1}, \sum_{n=0}^\infty {\rm d}^\mu_{n+2}\bigr)$ be a~solution to equation~\eqref{eqnsol} constructed in the proof of Proposition~{\rm\ref{propSol}}, and let the pair $(\rho_{1},\mu_{2})$ be its initial data.
The following statements are equivalent:
\begin{enumerate}\itemsep=0pt
\item[$(i)$] Equation $\bigl({\rm d}^\mu+{\rm d}^\rho\bigr)^2=0$ admits a~solution via $\left(\sum_{n=0}^\infty {\rm d}^\rho_{n+1}, \sum_{n=0}^\infty {\rm d}^\mu_{n+2}\right)$. %\label{con1}
\item[$(ii)$] The pair $(\rho_{1},\mu_{2})$ satisfies the conditions%\label{con2}
\begin{gather}
\bigl[{\rm d}^\mu_2+{\rm d}^\rho_2, {\rm d}^\rho_1\bigr]=0,\label{enueqn1}\\
\frac12\bigl[{\rm d}^\mu_2,{\rm d}^\mu_2\bigr]+{\rm d}^\rho_1\bigl({\rm d}^\mu_3\bigr)=0.\label{enueqn2}
\end{gather}
\item[$(iii)$] The pair $(\rho_{1},\mu_{2})$ defines the anchor and torsion, respectively, of a~Lie algebroid structure on $V$. %\label{con3}
\end{enumerate}
\end{Proposition}

\begin{proof}
The nontrivial steps involve showing that $(i)\Leftarrow~(ii)\Rightarrow~(iii)$.
Firstly, we show that statement $(ii)$ implies $(iii)$.
On $C^1( \mathfrak h)$, equation~\eqref{enueqn1} gives
\begin{gather*}
 -\mu_{2+0}(v_b, v_c )^a \rho_{1+0} (v_a)^i - \rho_{1+0} (v_b)^k\rho_{1+1} (y_k, v_c)^i \\
\qquad{} + \rho_{1+0} (v_c)^k \rho_{1+1} (y_k, v_b)^i =0,\qquad\forall  b,c,i.
\end{gather*}
Consider
$C_{bc}\,^a\equiv (\Gamma')_{i}\,_c\,^a (\rho)_b\,^i - (\Gamma')_{i}\,_b\,^a (\rho)_c\,^i- \mu_{bc}\,^a$,
this defines the structure constants of an $\mathbb R$-bilinear skew-symmetric bracket operation on $\Gamma(V)$ by $[v_b, v_c ] = C_{bc}^a v_a$, and $[f \alpha, \beta ] = f[\alpha, \beta ]-\rho(\beta)(f)\alpha $.
So equation~\eqref{enueqn1} gives the compatibility between the bracket operation on~$\Gamma(T_M)$ and that of $\Gamma(V)$,
\[ [\rho(v_b), \rho(v_c)]=\rho ( \nabla'_{\rho (v_b)} v_c-\nabla'_{\rho (v_c)} v_b-\mu(v_b, v_c ) ) = \rho ( [v_b, v_c ] ).\]

Equation~\eqref{enueqn2} gives
\begin{gather*}
\begin{split}
&\left(\frac12\mu_{0+2} \right)^a_{ b,c} v^b \wedge v^c\wedge \frac\partial{\partial v^a} \left( (\frac12\mu_{0+2} )^{\rm d}_{ e,f} v^e \wedge v^f\right) \\
& \qquad{}+(\rho_{0+1} )^i_{a} v^a \wedge\left(\frac12(\mu_{1+2} )^{\rm d}_{ i,b,c} v^b \wedge v^c\right) =0.
\end{split}
\end{gather*}
This is, after taking the dual, equivalent to the condition
\[\sum_{a,b,c{\rm\,cyclic}} \mu(\mu (v_a, v_b), v_c )
+\bigl({\nabla'}_{\rho (v_a)}^{(1)}\mu\bigr)(v_b,v_c)-R_{\nabla'}(\rho (v_a),\rho(v_b))(v_c) =0,\]
which is the Bianchi identity for Lie algebroids (see \cite[equation (61)]{Fernandes}). Note that although~$\mu_{1+2} $ depends on the gauge conditions\footnote{See Appendix~\ref{appB} for an explicit formula in arbitrary gauge.} coming from $l_2$ and $l'_2$, such gauge dependence does not survive under antisymmetrization of inputs in $\Gamma(V[1])$.

It remains to show that statement $(ii)$ implies statement $(i)$.
By the nature of derivation, this is true iff $\bigl({\rm d}^\mu+{\rm d}^\rho\bigr)^2\bigl(v^a\bigr) = \bigl({\rm d}^\mu+{\rm d}^\rho\bigr)^2\bigl(y^i\bigr) = 0$ for all $a\in \{1,\dots, \mathrm{rk}(V)\}$ and $i\in \{1,\dots, \dim (M)\}$.
For higher brackets, one can apply a~recursive argument as before, which we now describe briefly.

Consider elements such as
\[\sum_{p+q-1=n+4} \frac12 \bigl[{\rm d}^\mu_{p}+{\rm d}^\rho_{p}, {\rm d}^\mu_{q}+{\rm d}^\rho_{q}\bigr]\bigl(v^a\bigr)\qquad \text{and} \qquad \sum_{p+q-1=n+3} \frac12 \bigl[{\rm d}^\mu_{p}+{\rm d}^\rho_{p}, {\rm d}^\mu_{q}+{\rm d}^\rho_{q}\bigr]\bigl(y^i\bigr),
\]
where $n\geq0$.
By degree analysis, they are in the component \[\Omega^0\otimes \Gamma\bigl(\widehat{\operatorname{Sym}}^{\geq1} \bigl( T^\vee_M\bigr) \otimes \widehat{\operatorname{Sym}}\bigl( V^\vee[-1]\bigr)\bigr).\] As a~linear map, $\delta$ is invertible when restricted to the above component. So showing the vanishing of those elements is the same as showing that
\begin{gather*}
\sum_{p+q-1=n+4} \frac12 \delta\left( \bigl[{\rm d}^\mu_{p}+{\rm d}^\rho_{p}, {\rm d}^\mu_{q}
+{\rm d}^\rho_{q}\bigr]\bigl(v^a\bigr)\right)
=\sum_{p+q-1=n+4} \frac12 \bigl[\delta, \bigl[{\rm d}^\mu_{p}+{\rm d}^\rho_{p}, {\rm d}^\mu_{q}+{\rm d}^\rho_{q}\bigr]\bigr]\bigl(v^a\bigr) =0
\end{gather*}
and
\begin{gather*}
\sum_{p+q-1=n+3} \frac12 \delta\left( \bigl[{\rm d}^\mu_{p}+{\rm d}^\rho_{p}, {\rm d}^\mu_{q}+{\rm d}^\rho_{q}\bigr]\bigl(y^i\bigr) \right)
= \sum_{p+q-1=n+3} \frac12 \bigl[\delta, \bigl[{\rm d}^\mu_{p}+{\rm d}^\rho_{p}, {\rm d}^\mu_{q}+{\rm d}^\rho_{q}\bigr]\bigr]\bigl(y^i\bigr) = 0.
\end{gather*}
By the (graded) Jacobi identity of the derivation operators and equation~\eqref{eqnsol}, the commutator of type $\bigl[\delta,\bigl[{\rm d}^\mu_{p}+{\rm d}^\rho_{p}, {\rm d}^\mu_{q}+{\rm d}^\rho_{q}\bigr]\bigr]$ can be rewritten as a~sum of terms such as $-\bigl[{\rm d}^\nabla_s+{\rm d}^{\nabla'}_s,\bigl[{\rm d}^\mu_{t}+{\rm d}^\rho_{t}, {\rm d}^\mu_{u}+{\rm d}^\rho_{u}\bigr]\bigr]$, with $t+u\leq p+q-1$. Now an inductive argument can be applied here.
\end{proof}

The proof for Theorem~\ref{Theorem~1} follows from Propositions~\ref{propSol} and~\ref{prop:quadrel}.

\begin{Example} A~weak Lie algebra bundle is a~Lie algebroid with a~vanishing anchor. If $\rho\colon V\to T_M$ vanishes, so are $\{\rho_{n+1}\}_{n=0}^\infty$ by the recursive relation.
 A~(strong) Lie algebra bundle is a~weak Lie algebra bundle whose Lie algebra structure is fixed with respect to local trivializations. As shown in~\cite{CF}, this is equivalent to the condition that there exists a~connection $\nabla'$ on $V$ such that
\begin{gather*}
\nabla'[-,-] - [\nabla'(-),-]-[-,\nabla'(-)]
=\nabla'\mu(-,-) - \mu(\nabla'(-),-) - \mu(-,\nabla'(-))=0.
\end{gather*}
Therefore, $\mu_{1+2}$, which depends on the covariant derivative of $\mu$, vanishes. Furthermore, in the~$\delta^{-1}$ gauge, for any $v^a$,
\[{\rm d}^\mu_{2+2}\bigl(v^a\bigr)=-\delta^{-1}\left({\rm d}^{\nabla'}_2\bigl({\rm d}^\mu_{2}\bigl(v^a\bigr)\bigr)+ {\rm d}^{\mu}_2\bigl({\rm d}^{\nabla'}_{2}\bigl(v^a\bigr)\bigr)\right)=0.\]
Similarly, one can show that all operations in $\{\mu_{n+2}\}_{n=1}^\infty$ vanish in $\delta^{-1}$ gauge by an inductive argument.
\end{Example}

\subsection{Family version and homotopy equivalence}

By previous results, a~split Abelian extension $L_\infty$ structure on $\Omega_M\otimes \Gamma\bigl( T_M[-1]\oplus V\bigr)$ is determined by the choice of connections $(\nabla, \nabla')$ on~$T_M\oplus V[1]$ together with a~gauge condition on the~$\delta^{-1}$ image of CE differential. It can be inferred that any such $L_\infty$ structure on the same underlying space is ultimately homotopic; see~\cite{Co11, GG_ahat} for the explanation from the de Rham space perspective. We develop a~more explicit approach, which is particularly suited to our discussion of the homotopy type of geometric deformation of the split Abelian extension. Since the proof is somewhat long, we defer it to Appendix~\ref{appA}, see Propositions~\ref{htpyeqnhstru} and~\ref{htpyLinftystru}.

Building on this, in the geometrically deformed case,
the correspondence established in Theorem~\ref{Theorem~1} extends naturally to the homotopy category. Specifically, the map between geometric deformations and Lie algebroid structures respects homotopy equivalence of $L_\infty$ algebras. We demonstrate this by analyzing the homotopy type of the Fedosov differential governing the family of $L_\infty$ structures.

\begin{Proposition}[homotopy invariance of geometric deformation]\label{prophtpy}
Any two geometrically deformed $L_\infty$ structures on $\Omega\otimes \Gamma\bigl( T_M[-1]\oplus V\bigr)$ with the same underlying Lie algebroid structure are
$n$-homotopic for arbitrary $n\geq0$.
\end{Proposition}

If one stays in $\delta^{-1}$ gauge and works in the contexts of infinite jets, the result follows from the formal exponential map as in~\cite{GG_l}. The homotopy given in that context comes necessarily from a~genuine isomorphism of dg algebras. We consider the more computable Fedosov's approach, with arbitrary gauge choice, by working with a~family version of $L_\infty$ algebras. The main result is Proposition~\ref{gerhtpy}, from which Proposition~\ref{prophtpy} can be deduced.

Before giving the proof, we start by describing the geometrically deformed $L_\infty$ structure parameterized by the $n$-simplex $\Delta^n$.
Consider a~family of Lie algebroid structures and connections, $(\rho_t, [-,-]_t, \nabla_t, \nabla'_t)$, where $t\in \Delta^{n}$, and consider a~family of gauge conditions $(\phi_t,\psi_t)\in \Gamma\bigl( \operatorname{Sym}^{\geq3}\bigl(T_M^\vee \bigr)\otimes T_M \bigr)\oplus\Gamma\bigl( \operatorname{Sym}^{\geq2}\bigl(T_M^\vee \bigr)\otimes V^\vee[-1]\otimes V[1] \bigr)$.
By Lemma~\ref{familylem} in Appendix~\ref{appA}, Propositions~\ref{propSol} and~\ref{prop:quadrel}, there exists a~$t$-parameterised CE differential
\smash{$\delta + {\rm d}_{\geq1}^{\nabla_t} + {\rm d}_{\geq1}^{\nabla'_t}+{\rm d}^{\rho_t} +{\rm d}^{\mu_t}$},
such that
\[
{\rm d}_{1}^{\nabla_t}=\nabla_t , \qquad {\rm d}_{1}^{\nabla'_t}=\nabla'_t ,\qquad
 \delta^{-1}\bigl({\rm d}_{\geq2}^{\nabla_t} + {\rm d}_{\geq2}^{\nabla'_t}\bigr) = \phi_t+\psi_t, \qquad
\delta^{-1}\bigl({\rm d}^{\rho_t} +{\rm d}^{\mu_t}\bigr)=0.
\]
Now the total differential
\smash{${\rm d}_t +F(t)+ \delta + {\rm d}_{\geq1}^{\nabla_t} + {\rm d}_{\geq1}^{\nabla'_t}+{\rm d}^{\rho_t} +{\rm d}^{\mu_t}$}
on
\[\Omega_{\Delta^{n}}\otimes_{\mathbb R} \Omega_M\otimes \Gamma\bigl( \operatorname{Sym}^{m}\bigl(T_M^\vee \oplus V^\vee[-1]\bigr) \bigr)\]
needs to satisfy the flatness condition.

As a~derivation operator, $F(t)$ is an element of degree $1$ in
\[\Omega^1_{\Delta^{n}}\otimes_{\mathbb R} \Omega_M\otimes \Gamma\bigl( \widehat{\operatorname{Sym}}\bigl(T_M^\vee \oplus V^\vee[-1]\bigr) \otimes (T_M \oplus V[1])\bigr).\]

By degree analysis, one may write down the decomposition
$F(t) = F^0(t)+F^1(t)$, where
\begin{gather*}
\begin{split}
& F^0(t)\in \Omega^1_{\Delta^{n}}\otimes_{\mathbb R} \Gamma \bigl(\widehat{\operatorname{Sym}} \bigl( T_M^\vee\bigr) \otimes V^\vee[-1]\otimes V[1] \oplus \widehat{\operatorname{Sym}} \bigl( T_M^\vee\bigr) \otimes T_M\bigr),\\
& F^1(t)\in \Omega^1_{\Delta^{n}}\otimes_{\mathbb R} \Omega_M^1\otimes \Gamma\bigl( \widehat{\operatorname{Sym}} \bigl( T_M^\vee\bigr) \otimes V[1]\bigr).
\end{split}
\end{gather*}
However, the component $F^1(t)$ is not compatible with the natural surjective map $\Omega_{\Delta^{n}}\otimes_{\mathbb R}C^*(\mathfrak h\oplus \mathfrak v)\to \Omega_{\Delta^{n}}\otimes_{\mathbb R}C^*(\mathfrak h)$, and therefore should not be allowed in the extended differential of geometric deformations.
The consistency equation decomposes by form degrees,
\begin{gather}
\bigl[{\rm d}_t+ F^0(t), {\rm d}_{\geq1}^{\nabla_t} + {\rm d}_{\geq1}^{\nabla'_t}\bigr] +\bigl[ F^0(t), \delta\bigr]=0,\label{solhtpy1}\\
\bigl[{\rm d}_t+ F^0(t), {\rm d}^{\rho_t} +{\rm d}^{\mu_t}\bigr]=0,\label{checkhtpy1}\\
{\rm d}_t F^0 (t)+\frac12\bigl[ F^0(t), F^0(t)\bigr]=0.\label{checkhtpy2}
\end{gather}

Locally it helps to write
\begin{align*}
F^0(t) ={}& f^i \frac{\partial}{\partial y^i}+
f^i_{j}(t)y^j \frac{\partial}{\partial y^i}+
\frac12f^i_{jk}(t)y^j y^k \frac{\partial}{\partial y^i}+\cdots
+f^a_{b}(t) v^b\frac{\partial}{\partial v^a}+f^a_{j,b}(t) y^jv^b\frac{\partial}{\partial v^a}+\cdots\\
\equiv{}& F^0_0(t)+ F^0_{1,0}(t)+F^0_{2,0}(t)+ \cdots +F^0_{0,1}(t)+ F^0_{1,1}(t)+\cdots,
\end{align*}
where the component $F^0_{0}$ must vanish because it would be incompatible with the required family structure on $\Delta^n$. As a~result, $F^0_{1,0}=0$. By equation~\eqref{checkhtpy1}, the component $F^0_{0,1}$
 defines a~flat connection on the trivial $C^*(\mathfrak h\oplus\mathfrak v)$-bundle over $\Delta^n$. Since $\Delta^n$
 is contractible, every flat connection is gauge-equivalent to the trivial one. Hence, after applying a~global bundle automorphism (a gauge transformation), we may assume that this component vanishes. So effectively, the homotopy $F^0(t)$ can be realised as an element in
\[\Omega^1_{\Delta^{n}}\otimes_{\mathbb R} \Gamma \bigl(\widehat{\operatorname{Sym}}^{\geq1}\bigl( T_M^\vee\bigr) \otimes V^\vee[-1]\otimes V[1] \oplus \widehat{\operatorname{Sym}}^{\geq2} \bigl( T_M^\vee\bigr) \otimes T_M\bigr).\]

By the recursive nature of equation~\eqref{solhtpy1}, the other terms are entirely determined by the family connections $(\nabla_t, \nabla'_t)$ and gauge conditions $\phi_t$ and $\psi_t$, following a~similar analysis to Lemma~\ref{familylem} in Appendix~\ref{appA}. Since the equation involves no Lie algebroid related structures, the homotopy $F^0(t)$ is fundamentally incapable of encoding deformations of the latter.

Equations~\eqref{checkhtpy1} and~\eqref{checkhtpy2} impose the consistency conditions needed to be satisfied by~$F^0(t)$, which contain an infinite number of homogeneous components. However, most are satisfied automatically except for the first few, as we shall see in the following Lemma~\ref{F1van}.

To facilitate future discussions, we assign ${^\mathfrak h}$weight $k$ to the components of $F^0$ in
\[ \Omega^1_{\Delta^{n}}\otimes_{\mathbb R} \Gamma \bigl(\widehat{\operatorname{Sym}}^k \bigl( T_M^\vee\bigr) \otimes V^\vee[-1]\otimes V[1] \oplus \widehat{\operatorname{Sym}} ^{k+1}\bigl( T_M^\vee\bigr) \otimes T_M\bigr).\]
\begin{Lemma}\label{F1van}
Suppose $F^0(t)$ is obtained recursively from equation~\eqref{solhtpy1}. Let $G(k)$ be the $^{\mathfrak h}$weight-$k$-component of equation~\eqref{checkhtpy1} or of equation~\eqref{checkhtpy2}. For any $k\geq1$, if all $G(m)$ with $0\leq m<k$ are satisfied, then $G(k)$ is solved automatically.
\end{Lemma}

The proof follows a~strategy analogous to Proposition~\ref{prop:quadrel} and Lemma~\ref{familylem}. We begin by noting that the operator $\delta$ acts invertibly on the relevant equations, meaning it suffices to verify that the commutator $[\delta, -]$ annihilates the left-hand side of the equation $G(m)$. The key step is to express this commutator in terms of other commutators involving \smash{${\rm d}^{\nabla_t}_{\geq1} + {\rm d}^{\nabla'_t}_{\geq1}$}, an operator whose action increases the $^\mathfrak{h}$weight. This manipulation yields a~new expression that contains~$G(m)$ terms with a~lower $^\mathfrak{h}$weight, allowing us to conclude by applying the inductive hypothesis.

The $^{\mathfrak h}$weight-$0$-component of equation~\eqref{checkhtpy1} gives
$\bigl[{\rm d}_t, {\rm d}_1^{\rho_t}\bigr]=0$
and
$\bigl[{\rm d}_t, {\rm d}_2^{\mu_t}\bigr]+{\rm d}_1^{\rho_t}\bigl(F^0_{1,1}\bigr)=0$.
This is equivalent to the requirement that the $t$-dependence of the Lie algebroid structure on $V$ is trivial, by a~straightforward calculation.
The $^{\mathfrak h}$weight-$0$-component of equation~\eqref{checkhtpy2} is trivial, which immediately guarantees that the equation holds by Lemma~\ref{F1van}. Indeed, one checks that the lowest $^{\mathfrak h}$weight component of equation~\eqref{checkhtpy2} gives
${\rm d}_t F_{1,1}^0 (t)=0$
and
${\rm d}_t F_{2,0}^0 (t)=0$,
which can be verified explicitly. The above analysis is concluded in the following proposition.

\begin{Proposition}\label{gerhtpy}
Given a~family of connections $(\nabla_t, \nabla'_t)$ of $T_M\oplus V$ parameterised smoothly by $t\in \Delta^{n}$, and fix a~Lie algebroid structure on $V$, there exists a~flat $\Delta^n$-family of geometric deformations of $L_\infty$ algebra structures on $\Omega_M\otimes \Gamma\bigl( T_M[-1]\oplus V\bigr)$. Conversely, any flat $\Delta^n$-family of geometrically deformed $L_\infty$ algebra structures on $\Omega_M\otimes \Gamma\bigl( T_M[-1]\oplus V\bigr)$ given by a~flat connection ${\rm d}_t +F(t)$ with $F(t)$ in \[
\Omega^1_{\Delta^{n}}\otimes_{\mathbb R} \Gamma \bigl(\widehat{\operatorname{Sym}} ^{\geq1}\bigl( T_M^\vee\bigr) \otimes V^\vee[-1]\otimes V[1] \oplus \widehat{\operatorname{Sym}}^{\geq2} \bigl( T_M^\vee\bigr) \otimes T_M\bigr)
\]
necessarily comes from a~single Lie algebroid structure on $V$.
\end{Proposition}

The fact that one can set $F^0_{1,0}$ and $F^0_{0,1}$ to zero is not surprising --- in a~similar case~\cite{Kapranov}, the flat connection which encodes the homotopy equivalences was given by pulling back ${\rm d}_t$ via an exponential map, and there again the constant and linear terms are absent.

\section{Atiyah class and character}
\label{atiyah}

Atiyah class plays a~crucial role in cotangent theories, particularly in encoding the exterior derivatives of the trilinear and higher terms in the action functional~\cite{Co11, GG_ahat}. Furthermore, it can be employed to construct character classes in various contexts. Here we will firstly investigate the relationship between Atiyah class of a~geometrically deformed $L_\infty$ algebra and its underlying Lie algebroid structure. By utilizing this connection, we then generate the Atiyah--Chern classes, which are relevant in the one-loop computation of certain intriguing field theories.

Consider an $L_\infty$ algebra $\mathfrak g$. Its CE cochain $C^*(\mathfrak g)$ can be viewed as the algebra of functions~$\mathcal O_{B\mathfrak g}$ over the space $\mathfrak g[1]\equiv B\mathfrak g$.\footnote{To be precise, it is an $L_\infty$ space in the sense of~\cite{Co11,GG_l}.} Let
$\Omega^1_{B\mathfrak g}$ denote the module of K\"{a}hler
differentials of $\mathcal O_{B\mathfrak g}$. The de Rham differential \smash{${\rm d}^{\rm DR}\colon \mathcal O_{B\mathfrak g}\to
\Omega^1_{B\mathfrak g}$} is the universal derivation.

We now consider vector bundles over $B\mathfrak g$. The $\mathcal O_{B\mathfrak g}$-module $C^*(\mathfrak g)\otimes_{\Omega_M} \mathfrak g[1]$ plays the role of tangent bundle $T_{B\mathfrak g}$. There is a~natural flat connection $\nabla_{B\mathfrak g}\colon C^*(\mathfrak g)\otimes \mathfrak g[1]\to \Omega^1_{B\mathfrak g}\otimes \mathfrak g[1]$ on $T_{B\mathfrak g}$ characterized by\footnote{Here the construction is slightly different from~\cite{GG_ahat}. We are realising the notion of connection in a~purely algebraic setting, where the splitting is naturally given. In the reference the authors considered the connection over an infinite jet bundle, whose local definition depends on a~choice of local trivialization of $T_M$.}
\[\nabla_{B\mathfrak g} (f\otimes s) = {\rm d}^{\rm DR} f \otimes s,\qquad
f\in C^*(\mathfrak g), \quad s\in\mathfrak g[1].\]

The Atiyah class of $\nabla_{B\mathfrak g}$ is a~class in the cohomology of $\Omega^1_{B\mathfrak g}\otimes \operatorname{End}(\mathfrak g[1])$ that measures the failure of $\nabla_{B\mathfrak g}$ to be compatible with the CE differential ${\rm d}^{\mathfrak g}$.
More precisely, extending
$\nabla_{B\mathfrak g}$ to~${\Omega^\bullet_{B\mathfrak g}\otimes\mathfrak g[1]}$, we define the representing cocycle of Atiyah class to be
\[ {\rm At}(\nabla_{B\mathfrak g}):= \nabla_{B\mathfrak g}\circ {\rm d}^{\mathfrak g}+ {\rm d}^{\mathfrak g} \circ \nabla_{B\mathfrak g}.\]
 The differential on $\Omega^1_{B\mathfrak g}\otimes \mathfrak g[1]\cong C^*(\mathfrak g)\otimes \mathfrak g^\vee[-1]\otimes \mathfrak g[1]$ is given by the CE differential of the $L_\infty$ module $\mathfrak g^\vee[-1]\otimes \mathfrak g[1]$, and we again denote it as ${\rm d}^{\mathfrak g}$.

\begin{Lemma}
Viewed as a~linear map $C^*(\mathfrak g)\otimes \mathfrak g[1]\to \Omega^1_{B\mathfrak g}\otimes \mathfrak g[1]$, the cochain ${\rm At}(\nabla_{B\mathfrak g})$ is ${\rm d}^{\mathfrak g}$-closed and $\nabla_{B\mathfrak g}$-flat.
\end{Lemma}

\begin{proof}
The differential ${\rm d}^{\mathfrak g}$ extends naturally from $C^*(\mathfrak g)\otimes M$ to $\Omega^1_{B\mathfrak g}\otimes M$ for any $L_\infty$ module~$M$. For linear maps, we have that
\[ {\rm d}^{\mathfrak g}\bigl({\rm At}(\nabla_{B\mathfrak g})\bigr) = {\rm d}^{\mathfrak g} \circ {\rm At}(\nabla_{B\mathfrak g}) - {\rm At}(\nabla_{B\mathfrak g})\circ {\rm d}^{\mathfrak g}=\bigl[{\rm d}^{\mathfrak g},\bigl[{\rm d}^{\mathfrak g},\nabla_{B\mathfrak g}\bigr]\bigr],\]
which vanishes by Jacobi identity.

The connection $\nabla_{B\mathfrak g}$ extends from $C^*(\mathfrak g)\otimes \mathfrak g[1]$ to $C^*(\mathfrak g)\otimes \mathrm{End}(\mathfrak g[1])$, and can be shown to be flat by local computation. So
\[
 \nabla_{B\mathfrak g}\bigl({\rm At}(\nabla_{B\mathfrak g})\bigr) = \nabla_{B\mathfrak g}\circ {\rm At}(\nabla_{B\mathfrak g}) - {\rm At}(\nabla_{B\mathfrak g})\circ \nabla_{B\mathfrak g}=0.\tag*{\qed}
\] \renewcommand{\qed}{}
\end{proof}

\subsection{Atiyah class of Lie algebroids}

Consider a~geometrically deformed $L_\infty$ algebra with underlying space being $\mathfrak g=\mathfrak h\oplus \mathfrak v$ as before.
Locally, let $\bigl\{x^i\bigr\}$ be coordinates on $M$, and let $\bigl\{y^i\bigr\}$ and $\{v^a\}$ be the chosen local frames of~$T^\vee_M$ and~$V^\vee$, respectively, as above. We write $Dy^i$ and $Dv^a$ for ${\rm d}^{\rm DR} y^i$ and ${\rm d}^{\rm DR} v^a$, respectively. Then~$\Omega^1_{B\mathfrak g}$ is locally the free
$\mathcal O_{B\mathfrak g}$-module generated by $\bigl\{Dy^i,Dv^a\bigr\}$. A~local computation on basis~${\{ 1\otimes y_i,\, 1\otimes v_a\mid i=1,\dots, \dim (M),\, a = 1,\dots, \mathrm{rk}(V) \}}$ of $ C^*(\mathfrak g)\otimes\mathfrak g[1]$ produces the homogeneous components of ${\rm At}(\nabla_{B\mathfrak g})$ as $C^*(\mathfrak g)$-linear maps which take values in $ C^*(\mathfrak g)\otimes \mathfrak g^\vee[-1] \otimes \mathfrak g[1]$.

Notice that
\begin{align*}
{\rm d}^{\mathfrak g}(1\otimes y_i) ={}& 1\otimes \nabla y_i +\sum_{n\geq1} \frac1{n!}y^{i_1}\cdots y^{i_{n}} \otimes (l_{n+1})_{i_1 i_2\cdots i_{n} i}\,^j y_j \\
&+\sum_{n\geq1} \frac1{(n-1)!}y^{i_1}\cdots y^{i_{n-1}} v^b \otimes (\rho_{n+1})_{i_1 i_2\cdots i_{n-1} i b}\,^j y_j\\
&+\sum_{n\geq2}\frac1{2!} \frac1{(n-2)!}y^{i_1}\cdots y^{i_{n-2}} v^a v^b \otimes (\mu_{n+1})_{i_1 i_2\cdots i_{n-2} i ab}\,^c v_c\\
&+ \sum_{n\geq0} \frac1{n!}y^{i_1}\cdots y^{i_{n}} v^a\otimes (l'_{n+2})_{i_1 i_2\cdots i_{n} i a}\,^b v_b,
\end{align*}
so
\begin{align*}
 \nabla_{B\mathfrak g}\circ {\rm d}^{\mathfrak g}(1\otimes y_i) ={}& \sum_{n\geq1} \frac1{(n-1)!}y^{i_1}\cdots Dy^{i_{n}} \otimes (l_{n+1})_{i_1 i_2\cdots i_{n} i}\,^j y_j \\
&+\sum_{n\geq2} \frac1{(n-2)!}y^{i_1}\cdots Dy^{i_{n-1}} v^b \otimes (\rho_{n+1})_{i_1 i_2\cdots i_{n-1} i b}\,^j y_j\\
&+\sum_{n\geq1} \frac1{(n-1)!}y^{i_1}\cdots y^{i_{n-1}} D v^b \otimes (\rho_{n+1})_{i_1 i_2\cdots i_{n-1} i b}\,^j y_j\\
&+\sum_{n\geq3}\frac1{2!} \frac1{(n-3)!}y^{i_1}\cdots Dy^{i_{n-2}} v^a v^b \otimes (\mu_{n+1})_{i_1 i_2\cdots i_{n-2} i ab}\,^c v_c\\
&+\sum_{n\geq2}\frac1{(n-2)!}y^{i_1}\cdots y^{i_{n-2}} Dv^a v^b \otimes (\mu_{n+1})_{i_1 i_2\cdots i_{n-2} i ab}\,^c v_c\\
&+ \sum_{n\geq1} \frac1{(n-1)!}y^{i_1}\cdots Dy^{i_{n}} v^a\otimes (l'_{n+2})_{i_1 i_2\cdots i_{n} i a}\,^b v_b\\
&+\sum_{n\geq0} \frac1{n!}y^{i_1}\cdots y^{i_{n}} Dv^a\otimes (l'_{n+2})_{i_1 i_2\cdots i_{n} i a}\,^b v_b.
\end{align*}
This gives the component
\begin{align*}
 {\rm At}(\nabla_{B\mathfrak g})(1\otimes y_i) ={}& Dy^{j} \otimes l_{2}(y_{j},y_i) +Dy^{j}v^b \otimes \rho_3(y_j, y_i, v_b)+Dv^b \otimes \rho_2(y_i, v_b)\\
&+\frac1{2!} Dy^j v^a v^b \otimes \mu_4(y_j,y_ i, v_ a, v_b)+ Dv^a v^b \otimes \mu_3( y_i, v_a, v_b)\\
&+ Dy^j v^a\otimes l'_{3}(y_j,y_ i, v_a)+ Dv^a\otimes l'_{2}(y_i,v_ a)+ \cdots,
\end{align*}
where ``$\cdots$'' stands for terms with dependence on \smash{$\bigl\{y^i\bigr\}_{i=1}^{\dim (M)}$}.

Similarly, by
\begin{align*}
{\rm d}^{\mathfrak g}(1\otimes v_a) ={}& 1\otimes \nabla' v_a +\sum_{n\geq1} \frac1{n!}y^{i_1}\cdots y^{i_{n}} \otimes (l'_{n+1})_{i_1 i_2\cdots i_{n} a}\,^b v_b \\
&+\sum_{n\geq0} \frac1{n!}y^{i_1}\cdots y^{i_{n}} v^b \otimes (\mu_{n+2})_{i_1 i_2\cdots i_{n} b a}\,^c v_c \\
&+1\otimes \rho(v_a)+\sum_{n\geq1} \frac1{n!}y^{i_1}\cdots y^{i_{n}} \otimes (\rho_{n+1})_{i_1 i_2\cdots i_{n} a}\,^j y_j
\end{align*}
and
\begin{align*}
\nabla_{B\mathfrak g}\circ {\rm d}^{\mathfrak g}(1\otimes v_a) ={}& \sum_{n\geq1} \frac1{(n-1)!} y^{i_1}\cdots Dy^{i_{n}} \otimes (l'_{n+1})_{i_1 i_2\cdots i_{n} a}\,^b v_b\\
& +\sum_{n\geq0} \frac1{(n-1)!} y^{i_1}\cdots Dy^{i_{n}} v^b \otimes (\mu_{n+2})_{i_1 i_2\cdots i_{n} b a}\,^c v_c \\
&+\sum_{n\geq0} \frac1{n!} y^{i_1}\cdots y^{i_{n}} Dv^b \otimes (\mu_{n+2})_{i_1 i_2\cdots i_{n} b a}\,^c v_c\\
&+\sum_{n\geq1}\frac1{(n-1)!} y^{i_1}\cdots Dy^{i_{n}} \otimes (\rho_{n+1})_{i_1 i_2\cdots i_{n} a}\,^j y_j.
\end{align*}
So
\begin{align*}
{\rm At}(\nabla_{B\mathfrak g})(1\otimes v_a) ={}& Dy^j \otimes l'_{2}(y_j,v_ a) + Dy^{j} v^b \otimes \mu_3(y_j,v_b,v_a) \\
&+ Dv^b \otimes \mu_2(v_b,v_a)+Dy^{j}\otimes \rho_2(y_j,v_a)+\cdots.
\end{align*}

\begin{Proposition}\label{atiyah0part}
The constant part of the Atiyah cocycle defined on the geometrically deformed~$L_\infty$ algebra~$\mathfrak g$ is given by
\begin{gather*}
{\rm At}(\nabla_{B\mathfrak g})_0 = Dy^i\otimes \begin{pmatrix}l_2(y_i,-) & \rho_{1+1}(y_i,-) \\0 & l'_2(y_i,-)\end{pmatrix}
+ Dv^a\otimes \begin{pmatrix}\rho_{1+1}(-,v_a) & 0 \\l'_2(-,v_a) & \mu_2(v_a,-)\end{pmatrix}
\end{gather*}
as a~linear map in $\mathrm{Hom}\bigl(\mathfrak g[1],\mathfrak g^\vee[-1]\otimes \mathfrak g[1]\bigr)$.
Particularly in $\delta^{-1}$ gauge, this is
\begin{gather*}
Dy^i\otimes \begin{pmatrix}\frac13 \iota_{y_i} R (-)+\frac13\iota_{(-)} R (y_i) & \nabla_{\rho(-)}(y_i)-\check{\nabla}^{\rm bas}_{(-)}(y_i) \\0 & \frac12\iota_{y_i} R'(-)\end{pmatrix}\\
\qquad{}+ Dv^a\otimes \begin{pmatrix} \nabla_{\rho(v_a)}-\check{\nabla}^{\rm bas}_{v_a} & 0 \\ \frac12\iota_{(-)} R'(v_a) & \nabla'_{\rho(v_a)}- \hat{\nabla}^{\rm bas}_{v_a}\end{pmatrix},
\end{gather*}
where $\check{\nabla}^{\rm bas}$ and $\hat{\nabla}^{\rm bas}$ are basic connections\footnote{We recall the definition of basic connections at the beginning of Section~\ref{ss:modular}.} on $T_M$ and $V$ induced by $\nabla$ and $\nabla'$.
\end{Proposition}

\subsection{The Weil algebra}
\label{weilalgsubsec}
We shall remain in $\delta^{-1}$ gauge till the end of this section. Consider $\mathfrak g$ as before. We first need a~more detailed description of the
differential algebra $\Omega_{B\mathfrak g}$.

The differential algebra $\Omega_{B\mathfrak g}$ is obtained from
$\mathcal O_{B\mathfrak g}$ by taking the exterior algebra of its module of
Kähler differentials. We define
\[
\Omega_{B\mathfrak g}
:=
\bigwedge_{\mathcal O_{B\mathfrak g}}
\Omega^1_{B\mathfrak g}.
\]
Thus $\Omega_{B\mathfrak g}$ is the analogue of the de Rham algebra of
$\mathcal O_{B\mathfrak g}$, with the universal derivation
\[
{\rm d}^{\rm DR}\colon\
\mathcal O_{B\mathfrak g}
\longrightarrow
\Omega^1_{\mathcal O_{B\mathfrak g}}
\]
extended to the usual graded derivation on
$\Omega_{B\mathfrak g}$. Explicitly, after choosing local coordinates~$\bigl\{x^i\bigr\}$ on $M$, local frames $\bigl\{y^i\bigr\}$ of $T^\vee_M$ and local frames $\{v^a\}$ of $V^\vee$, $\Omega_{B\mathfrak g}$ can be identified
with the completed graded-commutative algebra generated by
$
y^i$, $ v^a$, $ Dy^i$, $ Dv^a$,
over the base algebra~$(\Omega_M, d)$.

We extend the grading introduced below Theorem~\ref{Theorem~1} to
$\Omega_{B\mathfrak g}$ by assigning internal degree $1$ to $Dy^i$ and
internal degree $2$ to $Dv^a$.
This equips $\Omega_{B\mathfrak g}$ with a~bigrading, with respect to which the total differential on $\Omega_{B\mathfrak g}$ decomposes into horizontal and vertical parts, where
\[
{\rm d}_{//}
:=
\delta+{\rm d}^\nabla_{\geq1}+{\rm d}^{\nabla'},
\qquad
{\rm d}_{\bot}
:=
{\rm d}^\rho+{\rm d}^\mu+{\rm d}^{\rm DR}.
\]

\begin{Lemma}
On \smash{$\Omega^{\star\geq1}_{B\mathfrak g}$}, we have that
\begin{gather*}
{\rm d}_{\geq1}^{\nabla} = \nabla+
\sum_{n>0} \frac1{n!}\bigl( (l_{n+1} )^i\,_{i_1,i_2,\dots, i_{n+1}}\, y^{i_1} \cdots Dy^{i_{n+1}} \bigr) \frac{\partial}{\partial Dy^i},\\
{\rm d}^{\nabla'} = \nabla'+
\sum_{n>0} \biggl( \frac1{(n-1)!} (l'_{n+1} )^b\,_{i_1,i_2,\dots, i_{n+1}, a}\, y^{i_1} \cdots Dy^{i_{n}} v^a \\
\phantom{{\rm d}^{\nabla'} =}{}+ \frac1{n!} (l'_{n+1} )^b\,_{i_1,i_2,\dots, i_{n}, a}\, y^{i_1} \cdots y^{i_{n}} Dv^a \biggr) \frac\partial{\partial Dv^b},
\\
{\rm d}^\rho =- \sum_{n=0}\frac1{n!} \bigl((\rho_{n+1})^i\,_{j_1\cdots j_n a} y^{j_1}\cdots y^{j_n} Dv^a +(\rho_{n+2})^i\,_{j_1\cdots j_{n+1} a} y^{j_1}\cdots Dy^{j_{n+1}} v^a \bigr) \frac\partial{\partial Dy^i},
\\
{\rm d}^\mu = -\sum_{n=0}\frac1{n!} \bigl((\mu_{n+2})^a\,_{j_1\cdots j_n bc} y^{j_1}\cdots y^{j_n} Dv^b \wedge v^c\bigr) \frac\partial{\partial Dv^a} \\
\phantom{{\rm d}^\mu =}{}-\sum_{n=0}\frac1{n!} \frac1{2!} \bigl((\mu_{n+3})^a\,_{j_1\cdots j_{n+1} bc} y^{j_1}\cdots Dy^{j_{n+1}} v^b\wedge v^c \bigr) \frac\partial{\partial Dv^a}.
\end{gather*}
\end{Lemma}

Let $W$ denote the symmetric algebra
\smash{$\Gamma\bigl(\widehat{\operatorname{Sym}} \bigl( T^\vee_M[-1] \oplus V^\vee[-1]\oplus V^\vee[-2]\bigr)\bigr)$},
equipped with the grading inherited from the vertical grading of $\Omega_{B\mathfrak g}$. Then, as a~graded vector space, \smash{$\Omega_{B\mathfrak g}^{(p,q)}\cong \Omega_M^p\otimes\Gamma\bigl(\widehat{\operatorname{Sym}} T^\vee_M\bigr) \otimes W^q$}.

\begin{Lemma}\label{HEcontext}
There exists an $\mathbb R$-linear special deformation retract
\[ \xymatrix{ %
\left(W,0\right)
\ar@<2pt>[rr]^{\iota_0}
&&
\left(\Omega_{B\mathfrak g}, \delta\right)
\ar@<2pt>[ll]^{p_0}
\save +<7mm,0mm>\ar@(ur,dr)^{-\delta^{-1}}\restore
} ,\]
where $\iota_0$ is the natural inclusion of $W^i$ into the degree $(0,i)$ component for each $i$, and $p_0$ is the projection onto $W$, annihilating all terms involving $y^i$'s or ${\rm d}x^i$'s.
\end{Lemma}
\begin{proof}
The identities below follow directly from the definitions of
$\iota_0$, $p_0$, and $\delta^{-1}$; see also \mbox{\cite{Fedosov,GLL}} for the standard Fedosov contraction
\begin{gather*}
p_0\circ \iota_0={\rm id}, \qquad \iota_0\circ p_0={\rm id}-\delta\circ \delta^{-1}-\delta^{-1}\circ \delta,\\
\delta^{-1}\circ\iota_0 = 0 = p_0\circ \delta^{-1}, \qquad \delta^{-1}\circ \delta^{-1}=0.\tag*{\qed}
\end{gather*} \renewcommand{\qed}{}
\end{proof}

\begin{Proposition}\label{propJets1}
There exists an $\mathbb R$-linear special deformation retract
\[ \xymatrix{ %
\left(W,0\right)
\ar@<2pt>[rr]^(.35){\iota}
&&
\left(\Omega_{B\mathfrak g}, {\rm d}_{//}\right)
\ar@<2pt>[ll]^(.65){p}
\save +<10mm,0mm>\ar@(ur,dr)^{h}\restore
} ,\]
which sends a~section from the left-hand side to a~${\rm d}_{//}$-cocycle in $\Omega_{B\mathfrak g}$.
\end{Proposition}
\begin{proof}
Consider the special deformation retract as in Lemma~\ref{HEcontext}, and consider a~perturbation of $\delta$ by ${\rm d}_{\geq1}\equiv {\rm d}^{\nabla}_{\geq1}+{\rm d}^{\nabla'}_{\geq1}$. By the homological perturbation lemma~\cite{Crainic04}, there exists a~perturbed special deformation retract, with
\[\iota = \sum_{n=0}^\infty \bigl(-\delta^{-1}\circ {\rm d}_{\geq1}\bigr)^n\circ \iota_0, \qquad p=p_0,\qquad h=-\delta^{-1}\circ \sum_{n=0}^\infty \bigl(-{\rm d}_{\geq1}\circ \delta^{-1}\bigr)^n,\]
such that ${\rm d}_{//}\circ \iota = 0$ is satisfied. So the image of $\iota$ is in the cocycle of $(\Omega_{B\mathfrak g}, {\rm d}_{//})$.
\end{proof}

It is useful to obtain an~explicit form of the embedding $\iota$.
\begin{Proposition}\label{propJets2}
Let \smash{$\bigl\{y^i\bigr\}_{i=1}^{\dim (M)}$} be a~local frame for $\Gamma\bigl( T^\vee_M\bigr)$ chosen as before.
Given any $f\in W$, there exists a~unique element $\tilde f$ in $\Omega_{B\mathfrak g}$ such that
\begin{itemize}
\itemsep=0pt
\item $\tilde f = f+\text{terms at least linear in $y^i$'s.}$, and
\item ${\rm d}_{//} \tilde f = 0$.
\end{itemize}
Furthermore, $\iota(f) = \tilde f$.
\end{Proposition}

\begin{proof}
The existence of $\tilde f$ is through a~recursive computation starting with
$\delta f = 0$ and $\delta (\text{terms linear in $y^i$'s})+\bigl({\rm d}^\nabla_1+{\rm d}^{\nabla'}_{1}\bigr) f=0$, etc. Notice that such $\tilde f$ is also unique, since $\delta$ is an isomorphism over components \[\Gamma\bigl(\widehat{\operatorname{Sym}}^{\geq1} \bigl( T^\vee_M\bigr)\otimes\widehat{\operatorname{Sym}} \bigl( T^\vee_M[-1] \oplus V^\vee[-1]\oplus V^\vee[-2]\bigr)\bigr).\] We claim that the map which sends $\forall f\in W$ to $\tilde f$ is precisely $\iota$. By the explicit formula,
\smash{$\iota(f) = f+\sum_{n=1}^\infty \bigl(-\delta^{-1}\circ {\rm d}_{\geq1}\bigr)^n(f)$},
where the second part contains terms at least linearly dependent on $y^i$'s. By the uniqueness, we have that $\iota(f)=\tilde f$.
\end{proof}

\begin{Example}
Propositions~\ref{propJets1} and~\ref{propJets2} work for arbitrary gauge.
In the $\delta^{-1}$ gauge, a~simplification for $\iota$ exists. Over the component \[\Gamma\bigl(\widehat{\operatorname{Sym}} \bigl( T^\vee_M\oplus T^\vee_M[-1] \oplus V^\vee[-1]\oplus V^\vee[-2]\bigr)\bigr),\]
the operator $\delta^{-1} \circ {\rm d}_{\geq1} = \delta^{-1}\circ \bigl({\rm d}^\nabla_1+{\rm d}^{\nabla'}_1\bigr)$.
So \smash{$\iota = \sum_{n=0}^\infty \bigl(-\delta^{-1}\circ \bigl({\rm d}^\nabla_1+{\rm d}^{\nabla'}_1\bigr)\bigr)^n\circ \iota_0$}.
In particular, we have that
\begin{gather*}
f(x)-\partial_i f(x)y^i+\cdots ,\qquad
 Dy^i-\Gamma^i_j\,_k y^j Dy^k+\cdots , \qquad v^a - \Gamma'^a_j\,_b y^j v^b+\cdots,\\
Dv^a - \Gamma'^a_j\,_b y^j Dv^b+\frac12R'_{jm}\,^a\,_b y^m Dy^j v^b+\cdots
\end{gather*}
are flat sections to ${\rm d}_{//}$, where we have omitted the terms containing quadratic or higher powers of~$y$.
\end{Example}

There is a~standard construction of the Weil algebra of a~Lie algebroid, as described in~\cite{AAC11,AAC}, whose underlying graded space coincides with the deformation retract $W$ constructed above.
We now compare its differential with the total differential on $\Omega_{B\mathfrak g}$.

\begin{Proposition}
The total complex $(\Omega_{B\mathfrak g}, {\rm d}_{//}+{\rm d}_{\bot})$ is homotopy equivalent to the Weil algebra of the Lie algebroid $V$ defined in~{\rm\cite{AAC11,AAC}}.
\end{Proposition}

\begin{proof}
Recall that in Proposition~\ref{propJets1}, we constructed a~special deformation retract of $(\Omega_{B\mathfrak g},\allowbreak {\rm d}_{//})$.
Now we regard ${\rm d}_\bot$ as a~perturbation of the differential ${\rm d}_{//}$.
By the homological perturbation lemma, the perturbed contraction induces a~differential on $W$ given formally by
\[
D_W = p\circ {\rm d}_{\bot}\circ \iota+p\circ {\rm d}_{\bot}\circ h\circ {\rm d}_{\bot}\circ \iota+\cdots.
\]
In the present situation, the explicit form of the contraction data implies
that all terms of order at least two in ${\rm d}_{\bot}$ vanish. Hence the induced differential reduces to
$
D_W=p\circ {\rm d}_{\bot}\circ \iota$.
The homotopy equivalence now follows.

It remains to identify the induced differential $D_W$ with the standard differential of Weil algebra. We denote the chosen local frames of $T^\vee_M[-1] $, $V^\vee[-1]$ and $V^\vee[-2]$ by \smash{$\bigl\{z^i\bigr\}_{i=1}^{\dim (M)}$}, \smash{$\{v^a\}_{a=1}^{\mathrm{rk}(V)}$} and \smash{$\{w^a\}_{a=1}^{\mathrm{rk}(V)}$}, respectively, then compute the local expressions of $D_W\bigl(z^i\bigr)$, $D_W\bigl(v^a\bigr)$ and $D_W\bigl(w^a\bigr)$.
The relevant differentials are
\begin{gather*}
\begin{split}
&{\rm d}^\rho = \rho^i_{a} v^a\frac\partial{\partial y^i} - \left(\rho^i_{a} Dv^a+ \rho^i_{j a} Dy^{j} v^a \right) \frac\partial{\partial Dy^i},\\
& {\rm d}^\mu = \frac12\mu^a_{ bc}v^b\wedge v^c \frac\partial{\partial v^a} - \left(\mu^a_{ bc} Dv^b \wedge v^c+ \frac1{2!}\mu^a_{j bc} Dy^{j} v^b\wedge v^c \right) \frac\partial{\partial Dv^a}.
\end{split}
\end{gather*}
So under the composition map $p\circ {\rm d}_{\bot} \circ \iota$ on $W$, one has that
\begin{align*}
Dy^i\mapsto{}& Dy^i-\Gamma^i_j\,_k y^j Dy^k+\cdots\\
\mapsto{}& - \bigl(\rho^i_{a} Dv^a+ \rho^i\,_{j a} Dy^{j} v^a \bigr) -\rho^j_{a} \Gamma^i_j\,_k v^a Dy^k \\
& - \bigl(\rho^i_{a} Dv^a - \bigl[-(\Gamma)_j^i\,_k (\rho)^k\,_a +\partial_j\bigl(\rho^i\,_a\bigr) + \rho^i\,_b(\Gamma')_j^b\,_a \bigr] Dy^{j} v^a \bigr) -\rho^j_{a} \Gamma^i_j\,_k v^a Dy^k +\cdots \\
\mapsto&- \rho^i_{a} Dv^a +\bigl(\partial_j\bigl(\rho^i\,_a\bigr) + \rho^i\,_b(\Gamma')_j^b\,_a ) Dy^{j} v^a \bigr),
\end{align*}
where the relation
$\rho^i\,_{j a} = (\Gamma)_j^i\,_k (\rho)^k\,_a -\partial_j\bigl(\rho^i\,_a\bigr) -\rho^i\,_b(\Gamma')_j^b\,_a$
has been used.

Similarly,
\begin{gather*}
v^a\mapsto v^a - \Gamma'^a_j\,_b y^j v^b+\cdots\mapsto Dv^a- \Gamma'^a_j\,_b Dy^j v^b+\left(\frac12\mu^a_{ bc} - \rho^j_{b} \Gamma'^a_j\,_c \right) v^b\wedge v^c+\cdots \\
\phantom{v^a}{}\mapsto Dv^a- \Gamma'^a_j\,_b Dy^j v^b-\frac12C_{bc}^av^b\wedge v^c,
\\
Dv^a\mapsto Dv^a - \Gamma'^a_j\,_b y^j Dv^b+\frac12R'_{jm}\,^a\,_b y^m Dy^j v^b+\cdots \\
\phantom{Dv^a}{}\mapsto -\left(\mu^a_{ bc} Dv^b \wedge v^c+ \frac1{2!}\mu^a_{j bc} Dy^{j} v^b\wedge v^c \right) - \rho^j_{b} \Gamma'^a_j\,_c v^b\wedge Dv^c\\
\phantom{Dv^a\mapsto}{}+\frac12 R'_{jm}\,^a\,_b \rho^m\,_c v^c\wedge Dy^j\wedge v^b- \Gamma'^a_j\,_b Dy^j Dv^b+\frac12R'_{jm}\,^a\,_b Dy^m Dy^j v^b+\cdots\\
\phantom{Dv^a}{}\mapsto -\bigl[(\Gamma')_{i}\,^a_c (\rho)^i\,_b -C_{bc}^a\bigr] Dv^b \wedge v^c- \Gamma'^a_j\,_b Dy^j Dv^b\\
\phantom{Dv^a\mapsto}{}+\frac12R'_{jm}\,^a\,_b Dy^m Dy^j v^b
+ \frac1{2!}R^{\rm bas}\,_{bcj}\,^a Dy^{j} v^b\wedge v^c.
\end{gather*}
We have used the relation
\[\mu^a_{j bc} = -\bigl(R^{\rm bas}(v_b, v_c)y_j\bigr)^a+ \frac12 R'_{ij}\,_b\,^a (\rho)_c\,^i-\frac12 R'_{ij}\,_c\,^a (\rho)_b\,^i,\]
which follows from the definition of basic curvature~\cite{AAC}
\[R^{\rm bas}(v_b, v_c)y_j=\nabla'_j[v_b, v_c]-[\nabla'_jv_b, v_c]-[v_b, \nabla'_jv_c]+\nabla'_{\nabla^{\rm bas}_{v_b} y_j} v_c-\nabla'_{\nabla^{\rm bas}_{v_c} y_j} v_b,\]
(where
$\nabla^{\rm bas}_{v_b} y_j = - \nabla_j\rho(v_b)+\rho(\nabla'_j v_b)+\nabla_{\rho(v_b)}y_j$
is the corresponding basic connection, as we shall recall at the beginning of Section~\ref{ss:modular}, see also~\cite{AAC}).
The calculation is straightforward using an explicit formula for $\mu_{3}$ given in Appendix~\ref{appB}.
\end{proof}

A~more functorial approach of the result was given in~\cite{GG_l}, where a~homotopy functor ${\rm enh}_{\rm mod}$ from the category of Lie algebroid representations up to homotopy to (dg-)vector bundles over~$L_\infty$ spaces was constructed. However, both the representations up to homotopy and the~$L_\infty$ structures rely on some choices. Our result shows that, if one fixes the connections on $T_M\oplus V$ once for all, a~one-to-one correspondence can be achieved on the nose. Moreover, by specifying different gauge conditions on the $L_\infty$ structure that encodes the Lie algebroid $V$, it is possible to induce more differentials on the Weil algebra $W$.

\begin{Remark}\label{weildr}
Since the vertical differential on $\Omega_{B\mathfrak g}$ decomposes into ${\rm d}^{\rm DR}$ and ${\rm d}^\mu+{\rm d}^\rho$, there is also a~decomposition on the differential $D_W$ over $W$, into differentials $D_W^{\rm DR}$ and $D_W^{\mu+\rho}$. In particular, locally
\begin{gather*}
D_W^{\rm DR}\bigl(Dy^i\bigr)=0, \qquad D_W^{\rm DR}\bigl(v^a\bigr)=Dv^a- \Gamma'^a_j\,_b Dy^j v^b,\\
D_W^{\rm DR}\bigl(Dv^a\bigr)=- \Gamma'^a_j\,_b Dy^j Dv^b+\frac12R'_{jm}\,^a\,_b Dy^m Dy^j v^b.
\end{gather*}
\end{Remark}

\subsection{Atiyah--Chern class}

The Atiyah class plays a~fundamental role in the construction of the Atiyah--Chern classes in the cohomology of $\Omega_{B\mathfrak g}$ through the Chern--Weil theory, whose representing cocycle is defined as~${c_k:=\mathrm{STr}\bigl({\rm At}(\nabla_{B\mathfrak g})^{\wedge k}\bigr)\in \Omega^k_{B\mathfrak g}}$,
where $\mathrm{STr}$ stands for the super-trace.

Since $\Omega_{B\mathfrak g}$ is bi-graded (as described at the beginning of Section~\ref{weilalgsubsec}), the $k$-th Atiyah--Chern class $c_k$ may have nontrivial components of degree $(p,q)$ for $p+q = k$. However, by the acyclicity of ${\rm d}_{//}$ in positive horizontal degrees, the class $c_k$ can be replaced by a~cohomologous representative concentrated in degree $(0,k)$ via a~homotopy-based transgression procedure. This representative gives rise to a~cocycle in the Weil algebra as stated in Proposition~\ref{propweilclass}.

\begin{Lemma}\label{transgr}
Given a~bicomplex $\bigl(C^{\bullet, \star}, {\rm d}_{//}, (-1)^\bullet {\rm d}_{\bot}\bigr)$ which arises from a~total complex with total differential $d = {\rm d}_{//}+ {\rm d}_{\bot}$, we consider a~cocycle $\mathbf{a}$ of total degree $k$. Suppose the horizontal differential ${\rm d}_{//}$ is acyclic on components $C^{\geq1, \star}$, then there exists a~representative $\alpha$ of degree~$(0, k)$ such that $\mathbf{a}-\alpha$ belongs to $d$-coboundaries.
\end{Lemma}

\begin{proof}
Without loss of generality, we can assume that $a$ has nontrivial components in degree~$(p, k-p)$ for each $0\leq p\leq k$, which we denote by $\mathbf{a}_{0,k}$, $\mathbf{a}_{1,k-1}, \dots, \mathbf{a}_{k,0}$, respectively. So we have, for $0\leq p\leq k-1$
\[{\rm d}_{//} \mathbf{a}_{k,0}=0, \qquad {\rm d}_{\bot} \mathbf{a}_{k-p,p}+{\rm d}_{//} \mathbf{a}_{k-p-1,p+1} = 0, \qquad {\rm d}_{\bot} \mathbf{a}_{0,k} = 0.\]
By the acyclicity of ${\rm d}_{//}$ on $C^{\geq1, \star}$, there exist elements $b_{0,k-1}, b_{1,k-2}, \dots, b_{k-1,0}$ such that for $0\leq p\leq k-2$
\[{\rm d}_{//} b_{k-1,0}=\mathbf{a}_{k,0}, \qquad {\rm d}_{//} b_{k-p-2,p+1} =\mathbf{a}_{k-p-1,p+1}- {\rm d}_{\bot} b_{k-p-1,p}.\]
That can be shown by an inductive method.
For $1\leq i<n$, suppose that there exists $b_{k-i-1, i}$ such that $\mathbf{a}_{k-i, i} = {\rm d}_{//} b_{k-i-1, i}+{\rm d}_{\bot} b_{k-i, i-1}$ holds.
When $i=n<k$,
\[{\rm d}_{//}(\mathbf{a}_{k-n,n}-{\rm d}_{\bot} b_{k-n, n-1}) = -{\rm d}_{\bot} \mathbf{a}_{k-n+1, n-1} +{\rm d}_{\bot} {\rm d}_{//}b_{k-n, n-1} =0,\]
and so there exists $b_{k-n-1,n}$ such that
${\rm d}_{//}b_{k-n-1,n} = \mathbf{a}_{k-n,n}-{\rm d}_{\bot} b_{k-n, n-1}$.
Finally, let $\alpha = \mathbf{a}_{0,k}-{\rm d}_\bot b_{0,k-1}$, and check that
\[\sum_{i=0}^k \mathbf{a}_{i,k-i} - \alpha = d\left(\sum_{i=0}^{k-1} b_{i,k-1-i}\right).\tag*{\qed}
\] \renewcommand{\qed}{}
\end{proof}

\begin{Proposition}\label{propweilclass}
Consider $\mathfrak g$ being the curved $L_\infty$ algebra arising from a~geometric deformation on $\mathfrak h\oplus \mathfrak v$. Suppose $u_{0,k}$ is a~total cocycle in $\Omega_{B\mathfrak g}$ concentrated in degree $(0,k)$, then its $y^i$-independent part is a~cocycle in the Weil algebra $(W,D_W)$.
\end{Proposition}

\begin{proof}
One checks that $u_{0,k}$ is both ${\rm d}_{//}$-closed and ${\rm d}_{\bot}$-closed. By Proposition~\ref{propJets1}, $y^i$-indepen\-dent part sits in the image of the projection map $p\colon \Omega_{B\mathfrak g}\to W$. Since the Weil differential is induced by ${\rm d}_{\bot}$ from $\Omega_{B\mathfrak g}$, it is a~cocycle.
\end{proof}

At first glance, this result may not seem particularly intriguing. It has been established that the Weil algebra is quasi-isomorphic to the de Rham cochain $(\Omega_M,{\rm d})$ \cite{AAC11}, meaning the only nontrivial classes land in Chern characters of $T_M$. However, this finding holds significance as it serves as an initial step towards developing an ``equivariant'' cohomology theory for Lie algebroids. It becomes interesting to examine captivating coboundaries within this context. To streamline the discussion, we will refer to the transgressed elements by their corresponding images in~$W$, without explicitly stating this fact.

Similarly, it is interesting to explore the component of the Atiyah cocycle independent of the coordinates $\bigl\{y^i\bigr\}$, formally expressed as ${\rm At}(\nabla_{B\mathfrak g})|_{y^i=0}$. This specific component (denoted by additional brackets compared to the constant component ${\rm At}(\nabla_{B\mathfrak g})_0$) encodes representation-theoretic information analogous to character classes.
\begin{align*}
{\rm At}(\nabla_{B\mathfrak g})|_{y^i=0} ={}& Dy^i\otimes \begin{pmatrix}l_2(y_i,-) & \rho_{1+1}(y_i,-) \\0 & l'_2(y_i,-)\end{pmatrix}+ Dv^a\otimes \begin{pmatrix}\rho_{1+1}(-,v_a) & 0 \\l'_2(-,v_a) & \mu_2(v_a,-)\end{pmatrix} \\
&+Dy^i v^a\otimes \begin{pmatrix}\rho_3(y_i, -, v_a) & 0 \\l'_3(y_i, -, v_a) & \mu_3(y_i,v_a,-)\end{pmatrix}\\
&+Dv^a v^b\otimes \begin{pmatrix}0 & 0 \\ \mu_3(-,v_a, v_b) & 0\end{pmatrix}
+Dy^i v^a v^b\otimes \begin{pmatrix}0 & 0 \\ \frac12\mu_4(y_i,-,v_a, v_b) & 0\end{pmatrix}\\
 ={}&Dy^i\otimes \begin{pmatrix}\frac13 {\rm d}x^m \left( R_{im}\,_k\,^j +R_{km}\,_i\,^j \right) & \rho_{1+1}(y_i,-) \\0 & \frac12 {\rm d}x^mR'_{im}\end{pmatrix}\\
&- Dv^a\otimes \begin{pmatrix}\check{\nabla}^{\rm bas}_{v_a} - \nabla_{\rho(v_a)} & 0 \\-l'_2(-,v_a) & \hat{\nabla}^{\rm bas}_{v_a} - \nabla'_{\rho(v_a)}\end{pmatrix} \\
&+Dy^i v^a\otimes \begin{pmatrix}\rho_3(y_i, -, v_a) & 0 \\l'_3(y_i, -, v_a) & \mu_3(y_i,v_a,-)\end{pmatrix}\\
&+Dy^i v^a v^b\otimes \begin{pmatrix}0 & 0 \\ \frac12\mu_4(y_i,-,v_a, v_b) & 0\end{pmatrix}.
\end{align*}

\subsection{Modular class and the first Atiyah--Chern class}
\label{ss:modular}

We recall the notion of basic connections from~\cite{AAC}. Let $V$ be a~Lie algebroid on a~smooth manifold~$M$, and let $\nabla$ and $\nabla'$ be connections on $T_M$ and on $V$, respectively. A~basic connection on bundle~${T_M[-1]\oplus V}$ induced by $(\nabla,\nabla')$ is a~$V$-connection $\check{\nabla}^{\rm bas}\oplus\hat{\nabla}^{\rm bas}$ on $T_M[-1]\oplus V$ such that
\begin{gather*}
\check{\nabla}^{\rm bas}_\alpha X = \nabla_{\rho(\alpha)}X- \nabla_X\rho (\alpha) +\rho\nabla'_X \alpha
=-\rho_{1+1}(X,\alpha)+\nabla_{\rho(\alpha)}X,\\
\hat{\nabla}^{\rm bas}_\alpha \beta = \nabla'_{\rho(\alpha)}\beta-\mu(\alpha, \beta) =
[\alpha,\beta]+\nabla'_{\rho(\beta)}\alpha,\end{gather*}
for all $X\in \Gamma(T_M)$ and $\alpha,\beta\in \Gamma(V)$.

So
\smash{$
 -\rho_{1+1}(X,\alpha)= \check{\nabla}^{\rm bas}_\alpha X - \nabla_{\rho(\alpha)}X$}, \smash{$
-\mu(\alpha, \beta) = \hat{\nabla}^{\rm bas}_\alpha \beta- \nabla'_{\rho(\alpha)}\beta$}.

By~\cite{Crainic00, CF}, there is a~Chern--Weil type of construction for secondary character classes from~$\smash{\check{\nabla}^{\rm bas}_\alpha X - \nabla_{\rho(\alpha)}X}$ and \smash{$\hat{\nabla}^{\rm bas}_\alpha \beta- \nabla'_{\rho(\alpha)}\beta$} via Chern--Simons transgression.
\begin{Proposition}\label{CSform1}
When $(\nabla, \nabla')$ is a~metric connection on $T_M[-1]\oplus V$, the expression
\[\mathrm{STr}\bigl( \check{\nabla}^{\rm bas} \oplus \hat{\nabla}^{\rm bas}- \nabla_{\rho(-)}\oplus \nabla'_{\rho(-)}\bigr) \]
computes the secondary character class $u_{1}$ in the Lie algebroid cohomology~$H^{1}(V)$.
\end{Proposition}

The (linear) connection \smash{$ \check{\nabla}^{\rm bas}\oplus \hat{\nabla}^{\rm bas}$} is equivalent to a~nonlinear flat connection on the superbundle $( T_M[-1]\oplus V,\rho)$ up to homotopy. The flatness is guaranteed by the ``bracket'' parts of \smash{$\check{\nabla}^{\rm bas}$} and \smash{$\hat{\nabla}^{\rm bas}$}, while the compatibility with respect to the anchor map $\rho$ makes sure the two components form a~connection for $( T_M[-1]\oplus V,\rho)$.

It is interesting to see whether higher secondary character classes \smash{$u_{\frac{k+1}2}$} as defined in~\cite{CF, Fernandes} can be recovered by the Atiyah class.

We now apply Lemma~\ref{transgr} to $c_1$. We start with the $(1,0)$-component of $c_1$, and apply the homotopy contraction $\delta^{-1}$,\footnote{Strictly speaking, the homotopy contraction for ${\rm d}_{//}\equiv \delta+{\rm d}^\nabla_{\geq1}+{\rm d}^{\nabla'}_{\geq1}$ is $h$ as given in Proposition~\ref{propJets1}, which includes additional linear operators with higher $^\mathfrak h$weights. However, for the leading weight computation, these corrections are irrelevant.} which gives
\begin{gather*}
\delta^{-1} Dy^i {\rm d}x^m \mathrm{STr} \begin{pmatrix}\frac13 \left( R_{im}\,_k\,^j +R_{km}\,_i\,^j \right) & \rho_{1+1}(y_i,-) \\0 & \frac12 R'_{im}\end{pmatrix} \\
\qquad= - y^m Dy^i \mathrm{STr} \begin{pmatrix}\frac13 \left( R_{im}\,_k\,^j +R_{km}\,_i\,^j \right) & 0 \\0 & \frac12 R'_{im}\end{pmatrix}.
\end{gather*}
Upon applying ${\rm d}_\bot$, the above expression
\begin{gather*}
\qquad\mapsto \bigl(- Dy^m-\rho^m\,_a v^a\bigr) \wedge Dy^i \mathrm{STr} \begin{pmatrix}\frac13 \left( R_{im}\,_k\,^j +R_{km}\,_i\,^j \right) & 0 \\0 & \frac12 R'_{im}\end{pmatrix}\\
\qquad= \bigl( Dy^i\wedge Dy^m+ \rho^m\,_a Dy^i v^a\bigr) \mathrm{STr} \begin{pmatrix}\frac13 \left( R_{im}\,_k\,^j +R_{km}\,_i\,^j \right) & 0 \\0 & \frac12 R'_{im}\end{pmatrix}\\
\qquad= Dy^i\wedge Dy^m \left(\frac13 R_{jm}\,_i\,^j \right)+ \rho^m\,_a Dy^i v^a \left(\frac13 R_{jm}\,_i\,^j \right).
\end{gather*}
The first term vanishes, since
\[
R_{j[mi]}\,^j =\frac12\bigl( -R_{mij}\,^j - R_{ijm}\,^j\bigr)-\frac12R_{jim}\,^j =-\frac12 R_{ijm}\,^j-\frac12R_{jim}\,^j=0.\]

By Lemma~\ref{transgr}, the new representative for the first Atiyah--Chern class is
\begin{gather*}
- \frac13 \rho^m\,_a Dy^i v^a R_{jm}\,_i\,^j + Dv^a\mathrm{STr} \begin{pmatrix}\rho_{1+1}(-,v_a) & 0 \\0& \mu_2(v_a,-)\end{pmatrix} \\
\qquad{}+Dy^i v^a\mathrm{STr} \begin{pmatrix}\rho_3(y_i, -, v_a) & 0 \\0 & \mu_3(y_i,v_a,-)\end{pmatrix}+\cdots.
\end{gather*}
Here $\cdots$ denotes terms of higher total degree in the formal coordinates $\bigl\{y^i\bigr\}$. It is precisely the leading-order part displayed above that survives under the projection onto $W$.

Let $C^*(V)$ denote the CE cochain for Lie algebroid $V$, which may also be regarded as a~subspace inside $W$, depending on the context.
\begin{Proposition}\label{modularc1}
If $\nabla'$ is a~metric connection on $V$, the first Atiyah--Chern class transgresses to the image of modular class~{\rm\cite{ELW}}
\[u_1\equiv \mathrm{STr} \begin{pmatrix} -\bigl(\check{\nabla}^{\rm bas} - \nabla_{\rho}\bigr) & 0 \\0 & -\bigl(\hat{\nabla}^{\rm bas} - \nabla'_{\rho}\bigr)\end{pmatrix}\in C^1(V)\]
under the de Rham differential $D_W^{\rm DR}$ in Weil algebra $W$ for Lie algebroid $V$.
\end{Proposition}
\begin{proof}
We have computed the degree-$(0,1)$ representative of $c_1$ and its image in $W$. So it remains to compare the image with
\[ D_W^{\rm DR} \left(v^a\cdot \mathrm{STr} \begin{pmatrix}\rho_{1+1}(-,v_a) & 0 \\0& \mu_2(v_a,-)\end{pmatrix}\right).\]

{\samepage According to the local form of $D_W^{\rm DR}$ given in Remark~\ref{weildr}, it is given by
\begin{gather*}
 \bigl(D v^a- \Gamma'^a_j\,_b Dy^j v^b\bigr)\otimes   \mathrm{STr} \begin{pmatrix}\rho_{1+1}(-,v_a) & 0 \\0& \mu_2(v_a,-)\end{pmatrix}\\
\qquad{}- Dy^j v^a\mathrm{STr}\begin{pmatrix}\nabla_j\rho_{1+1}(-,v_a) - \rho_{1+1}(\nabla_j(-),v_a) & 0 \\0& \nabla'_j\mu_2(v_a,-)-\mu_2(v_a,\nabla_j'(-))\end{pmatrix}.
\end{gather*}}

Note that
\begin{gather*}
- \Gamma'^a_j\,_b Dy^j v^b\otimes \mathrm{STr} \begin{pmatrix}\rho_{1+1}(-,v_a) & 0 \\0& \mu_2(v_a,-)\end{pmatrix} \\
\qquad=Dy^j v^a\otimes \mathrm{STr} \begin{pmatrix}\rho_{1+1}(-,\nabla'_jv_a) & 0 \\0& \mu_2(\nabla'_jv_a,-)\end{pmatrix}.
\end{gather*}

So in total we have
\begin{gather}
D v^a\otimes \mathrm{STr} \begin{pmatrix}\rho_{1+1}(-,v_a) & 0 \\0& \mu_2(v_a,-)\end{pmatrix}\nonumber\\
\qquad{}- Dy^j v^a \mathrm{STr}\begin{pmatrix}\nabla_j(\rho_{1+1})(-,v_a) & 0 \\0& \nabla'_j(\mu_{2})(-,v_a)\end{pmatrix},
\label{eqncompmod}
\end{gather}
where we have used $\nabla_j(\rho_{1+1})(-,v_a)$ to denote
$\nabla_j\rho_{1+1}(-,v_a) - \rho_{1+1}(\nabla_j(-),v_a)-\rho_{1+1}(-,\nabla_j'v_a)$,
and likewise $\nabla'_j(\mu_{2})(-,v_a)$ to denote
$\nabla_j'\mu_2(v_a,-)-\mu_2(v_a,\nabla'_j(-))-\mu_2(\nabla_j'v_a,-)$.

The second line in equation~\eqref{eqncompmod} gives\footnote{For more details, please refer to equation~\eqref{compmodeqn2} in Appendix~\ref{appB}.}
\[ Dy^kv^a \left( \operatorname{Tr} \rho_3(-, y_k, v_a) - \operatorname{Tr} \mu_3(y_k, v_a, -) \right)-\frac13 Dy^kv^a \rho^ i\,_a R_{jik}\,^j.\]
The proposition follows.
\end{proof}

The modular class $u_1$ is automatically $D_W^{\rho+\mu}$-closed as an element of $C^*(V)\subset W$. Consequently, within the Weil algebra, the first Atiyah--Chern representative is precisely $D_W u_1$. Thus the first Atiyah--Chern class is exact in the total cohomology of $W$, and of $\Omega_{B\mathfrak g}$. As the Weil algebra $W$ is quasi-isomorphic to the de Rham cochain over $M$ \cite{AAC11}, the modular class, despite being possibly nontrivial in Lie algebroid cohomology $H^1(V)$, is unlikely to survive in the ordinary cohomology on $M$.

It is possible to recognise more Atiyah--Chern classes due to the acyclicity of ${\rm d}^{\rm DR}$ in $\Omega_{B\mathfrak g}$, which might allow us to relate those classes to more secondary invariants of Lie algebroids in~$C^*(V)$ as defined in~\cite{Crainic00, CF}. However, in the absence of a~satisfactory
equivariant de Rham algebra in the present setting, we do not pursue
such identifications here.

\subsection{Action algebroids}

We consider the action Lie algebroid coming from a~finite dimensional connected Lie group action~${G\to \mathrm{Diff}(M)}$. In that case, the associated action Lie algebroid $V\equiv M\times \mathrm{Lie}(G)$ is trivial, whose anchor map is determined by the infinitesimal Lie algebra action $\mathrm{Lie}(G)\to \mathfrak X(M)$. It is possible to choose the trivial connection on $V$ such that the operations $\{l'_{n}\}_{n\geq2}$ encoding bundle structure vanish by induction.
Moreover, the torsion map $\mu$ is the $C^\infty$-linear extension of the Lie algebra structure constants on $\mathrm{Lie}(G)$ up to a~sign, hence all the higher brackets~$\{\mu_{2+n}\}_{n\geq1}$ vanish.
As a~result,
\begin{align*}
{\rm At}(\nabla_{B\mathfrak g})|_{y^i=0} ={}&Dy^i\otimes \begin{pmatrix}\frac13 {\rm d}x^m \left( R_{im}\,_k\,^j +R_{km}\,_i\,^j \right) & \rho_{1+1}(y_i,-) \\0 & 0\end{pmatrix}\\
&- Dv^a\otimes \begin{pmatrix}\check{\nabla}^{\rm bas}_{v_a} - \nabla_{\rho(v_a)} & 0 \\0 & \hat{\nabla}^{\rm bas}_{v_a} - \nabla'_{\rho(v_a)}\end{pmatrix} \\
&+Dy^i v^a\otimes \begin{pmatrix}\rho_3(y_i, -, v_a) & 0 \\0 & 0\end{pmatrix}.
\end{align*}
The component \smash{$-Dv^a\otimes \left(\check{\nabla}^{\rm bas}_{v_a} - \nabla_{\rho(v_a)}\right)$} is precisely the moment map for the infinitesimal $\mathrm{Lie}(G)$-action on bundle $T_M$, which was denoted by \smash{$\mathcal L^{ T_M}(\alpha)-\nabla_{\rho(\alpha)}$} for any $\alpha\in \mathrm{Lie}(G)$ in~\cite{BGV}, where~$\mathcal L^{ T_M}$ is the Lie derivative on $T_M$.

\begin{Proposition}\label{equivarKalk}
The $k$-th Atiyah--Chern cocycle of the action Lie algebroid has its constant part given by
\[
(-1)^k\mathrm{STr} \begin{pmatrix} R(-)+\mathcal L^{ T_M}_{(-)}(\alpha)-\nabla_{\rho(\alpha)}(-) & 0 \\0 & [\alpha,-]\end{pmatrix}^k,\qquad\forall \alpha\in \mathrm{Lie}(G)
\]
in the Cartan subcomplex \smash{$\big(\operatorname{Sym}^*\bigl(\mathrm{Lie}(G)^\vee\bigr)\otimes \Omega_M\bigr)^{\mathrm{Lie}(G)}$} of the BRST complex as in~{\rm\cite{Kalkman}}.
\end{Proposition}

The proof is obtained through a~direct calculation, focusing on the leading order contributions in powers of $v^a$'s and $y^i$'s. This greatly simplifies the procedure. In this case, only~${{\rm d}_\bot = {\rm d}^{\rm DR}}$ needs to be taken into account, and the contributions reside in $\operatorname{Sym}^*\bigl(\mathrm{Lie}(G)^\vee\bigr)\otimes \Omega_M$. Remarkably, the result agrees with the equivariant Chern character\footnote{Since the Lie group $G$ is assumed to be connected, the invariant space \smash{$\left(\operatorname{Sym}^*\bigl(\mathrm{Lie}(G)^\vee\bigr)\otimes \Omega_M\right)^{\mathrm{Lie}(G)}$} coincides with \smash{$\left(\operatorname{Sym}^*\bigl(\mathrm{Lie}(G)^\vee\bigr)\otimes \Omega_M\right)^G$}. The latter is the underlying graded space of the Cartan model for the $G$-equivariant cohomology of $M$.} of $T_M\ominus V$ using the Chern--Weil construction, leading to an immediate conclusion. Consequently, the Atiyah--Chern classes
serve as a~generalization of the equivariant Chern character for the
action Lie algebroid case.

Note that, intrinsically, these Atiyah--Chern classes are determined solely by the infinitesimal $\mathrm{Lie}(G)$-action, and their definition does not require an integrated Lie group action. Secondly, it would be intriguing to provide a~comprehensive
equivariant description of the de Rham algebra~$\Omega_{B(\mathfrak h\oplus \mathfrak v)}$ for the general Lie algebroid situation.

\section[L\_infty algebras in field theories]{$\boldsymbol{ L_\infty}$ algebras in field theories}
\label{aksz}

In this section, we show the possible applications of our probe in the areas of quantum field theory.

\subsection{Admissible pairings and BV field theories}

\begin{Definition}
An admissible pairing of degree $k-2$ on a~curved $L_\infty$ algebra $\mathfrak g$ over $(A,{\rm d}_A)$ is given by a~graded skew-symmetric pairing of degree $k$\footnote{Pairings of degree $k$ can be understood as maps of degree $k$ between dg modules. One such map is described either as a~degree $0$ map from the source, say $(M_1,{\rm d}_1)$, to the shifted target $(M_2[k], {\rm d}_2)$, or as a~degree $k$ map to the unshifted target $(M_2, {\rm d}_2)$. Within this paper, we shall use the latter description, unless otherwise noted.} on $\mathfrak g[1]$
\[\langle-,-\rangle\colon\ \mathfrak g[1]\otimes\mathfrak g[1]\to A\]
such that for every $l_n$ operation and for every $(n+1)$-tuple of homogeneous elements $v_1, \dots, v_{n+1}$ in $\mathfrak g[1]$, the multilinear map (of degree $k+1$)
$
\operatorname{Sym}^n \mathfrak g[1]\otimes \mathfrak g[1]\to A
$
given by $\langle l_n(v_1,\dots, v_n),v_{n+1}\rangle$ is graded symmetric.
\end{Definition}

Admissible pairings are particularly important for model building of classical BV theories in~\cite{AKSZ}.

\begin{Proposition}
Let $(A, {\rm d}_A)$ be a~commutative dg algebra over $\mathbb R$. Fix a~closed oriented manifold $\Sigma$ of dimension $m$ and consider an $L_\infty$ algebra over $(A, {\rm d}_A)$ with an admissible pairing $\langle-,-\rangle$ of degree $m-3$. There exists an $m$-dimensional classical BV theory\footnote{Unlike an ordinary scalar-valued functional, the action functional $S$ takes value in the commutative dg algebra~$(A,{\rm d}_A)$.} in the sense of~{\rm\cite{AKSZ, Co11}}, determined by the classical BV action
\[
S[\phi]=\int_\Sigma \left\langle \phi, \frac12 {\rm d}_\Sigma \phi +\sum_{n=0}^\infty \frac1{(1+n)!}l_n\bigl(\phi^{\otimes n}\bigr)\right\rangle, \qquad \phi\in \Omega_\Sigma\otimes_{\mathbb R}\mathfrak g[1].
\]
\end{Proposition}

\begin{Lemma}\label{lem:cotang}
Given a~$($curved$)$ $L_\infty$ algebra $\mathfrak g$ over $(A, {\rm d}_A)$, there exists an~extended $L_\infty$ algebra~${\mathfrak g\oplus\mathfrak g^\vee[k-2]}$ with a~non-degenerated admissible pairing of degree $k-2$ for any $k\in \mathbb Z$.
\end{Lemma}
\begin{proof}
Consider the minimal extension of $\mathfrak g$ by its coadjoint module $\mathfrak g^\vee[k-2]$. There exists a~natural pairing $\mathfrak g[1]\otimes \mathfrak g^\vee[k-1]\to A$ of degree $k$. It is possible to extend the pairing onto~${\bigl(\mathfrak g[1]\oplus \mathfrak g^\vee[k-1]\bigr)\wedge\bigl(\mathfrak g[1]\oplus \mathfrak g^\vee[k-1]\bigr)}$ graded skew-symmetrically. To make sure the extended pairing is admissible, it remains to check
\[\langle l_n(X_1,\dots, X_n), \alpha\rangle =(-1)^{|\alpha|\cdot |X_n|}\langle l_n(X_1,\dots,X_{n-1},\alpha),X_n\rangle\]
for all $X_1,\dots, X_n\in \mathfrak g[1]$ and $\alpha\in \mathfrak g^\vee[k-1]$.

Consider the $l_n(X_1,\dots, X_{n-1},-)$ as a~linear operator of degree $1+\sum_{i=1}^{n-1} |X_i|$. By the coadjointness of the $L_\infty$ action on $\mathfrak g^\vee[k-1]$,
\[\langle l_n(X_1,\dots, X_n), \alpha\rangle =-(-1)^{(1+\sum_{i=1}^{n-1} |X_i|)\cdot |X_n|}\langle X_n,l_n(X_1,\dots,X_{n-1},\alpha)\rangle.\]
By graded skew-symmetry of the extended pairing,
\[ \langle X_n,l_n(X_1,\dots,X_{n-1},\alpha)\rangle=-(-1)^{|X_n|\cdot(1+|\alpha|+\sum_{i=1}^{n-1} |X_i|)} \langle l_n(X_1,\dots,X_{n-1},\alpha),X_n\rangle.\]
This completes the proof.
\end{proof}

Combining the above two results, for any $L_\infty$ algebra $\mathfrak g$ and for any closed oriented manifold~$\Sigma$, there exists a~classical BV theory, whose equation of motion encodes the Maurer--Cartan equation for $\mathfrak g$. This is, in particular, an interesting way to encode Lie algebroid structure in such theories.

\begin{Proposition}\label{prop:cotang}
Given a~Lie algebroid $V\to M$, and given a~closed $m$-dimensional oriented manifold $\Sigma$, there exists a~classical BV theory on $\Sigma$ determined by the classical action
\[
S[\phi,\psi]=\int_\Sigma \left\langle \psi, {\rm d}_\Sigma \phi +\sum_{n=0}^\infty \frac1{(1+n)!}l_n\bigl(\phi^{\otimes n}\bigr)\right\rangle,
\]
where $\phi\in \Omega_{\Sigma}\otimes_{\mathbb R}\mathfrak g[1], \psi\in \Omega_{\Sigma}\otimes_{\mathbb R}\mathfrak g^\vee[m-2]$. The $L_\infty$ algebra $\mathfrak g:=\Omega_{M}\otimes_{C^\infty_M} (\Gamma(T_M)[-1]\oplus\Gamma(V))$ encodes the Lie algebroid structure on $V$. The $\Omega_M$-linear pairing $\int_\Sigma \langle-,-\rangle$ is $(-1)$-shifted symplectic and determines the BV bracket.
\end{Proposition}
In Lemma~\ref{lem:cotang} and hence also in Proposition~\ref{prop:cotang}, we take the shifted cotangent complex of the geometrically deformed $L_\infty$ algebra $\mathfrak g$ in order to obtain a~shifted symplectic pairing as required by the classical BV theories. Such field theories are referred to as cotangent theories\footnote{In general, there are more cotangent theories than the topological ones we showed here, depending on the geometric structures on the worldsheet manifold $\Sigma$.} in~\cite{Co11,GG_ahat}.
Quantization of cotangent BV theories has been studied extensively, for example, in aforementioned references, and relevant vacuum diagrams consist of only $n$-gons for each~${n\geq1}$. A~Feynman diagram analysis shows that the contribution associated with an $n$-gon is in one-to-one correspondence with the $n$-th Atiyah--Chern class of $\mathfrak g$.

\subsection{Poisson manifolds}

In the previous section, to obtain a~classical BV theory, we start with an $L_\infty$ algebra encoding a~Lie algebroid structure and take the shifted cotangent extension to construct an admissible pairing. Alternatively one could start with a~cotangent theory, whose underlying $L_\infty$ algebra~$\mathfrak h$ encodes only a~smooth structure, with a~fixed admissible pairing coming from taking the cotangents $\mathfrak h\oplus \mathfrak h^\vee[m-3]$. The classical Lagrangian looks exactly like the one in Proposition~\ref{prop:cotang}. The question is, can this $L_\infty$ algebra $\mathfrak h\oplus \mathfrak h^\vee[m-3]$ be further deformed, while preserving the admissible structure? In the case that is relevant to our main result, the underlying bundle structure of the $L_\infty$-module $\mathfrak h^\vee[m-3]$ needs to be of degree zero, so $m=2$. The answer leads to the following theorem.

\begin{Theorem}\label{posthm}
Consider the extended $L_\infty$ algebra $\mathfrak h_M\oplus \mathfrak h_M^\vee[-1]$ encoding the cotangent bundle over a~smooth manifold $M$. Together with the natural admissible pairing as guaranteed by Lemma~{\rm\ref{lem:cotang}}, there exists a~nontrivial geometric deformation on it only if $M$ has a~Poisson structure. Conversely, each Poisson manifold $M$ has a~unique geometric deformation of $\mathfrak h_M\oplus \mathfrak h_M^\vee[-1]$, which preserves the natural admissible pairing.
\end{Theorem}

\begin{proof}
Consider the abelian extension of $\mathfrak h_M$ by its coadjoint module $\mathfrak h_M^\vee[-1]$,
whose underlying space is given by $\bigl(\mathfrak h_M\oplus \mathfrak h_M^\vee[-1]\bigr)[1]\cong\Omega_M\otimes\Gamma\bigl(T_M[-1]\oplus T^\vee_M\bigr)[1]$. The natural pairing $\langle-,-\rangle$ between $\Gamma(T_M)$ and $\Gamma\bigl(T^\vee_M\bigr)$ is graded skew-symmetrically extended to the full space. As the~$L_\infty$ structure is entirely determined by choices of connections in these bundles, in order for the pairing to be invariant, it suffices to require that
\[\langle l_1(X),\alpha\rangle = (-1)^{|X|\cdot|\alpha|}\langle l_1(\alpha),X\rangle,\]
$\forall X\in \Omega_M\otimes \Gamma(T_M), \alpha\in \Omega_M\otimes \Gamma\bigl(T^\vee_M\bigr)[1]$. Locally this means that
$\Gamma_{li}\,^j = -\Gamma'_{l}\,^j\,_i$, $\forall i,j,l$.
To make sure the pairing is invariant with respect to the new structures $\{\rho_n,\mu_{n+1}\}_{n=1}^\infty$ obtained from geometric deformation, it is only necessary to solve the following two constraints on $\rho_1$, $\rho_2$ and $\mu_2$
\[
\langle \rho_1(\alpha), \beta \rangle = (-1)^{|\alpha|\cdot|\beta|} \langle \rho_1(\beta),\alpha\rangle,\qquad
\langle \rho_2(\alpha,X), \beta\rangle = (-1)^{|X|\cdot|\beta|}\langle \mu_2(\alpha,\beta),X \rangle
\]
for any elements $\alpha,\beta\in \Omega_M\otimes \Gamma\bigl(T^\vee_M\bigr)[1]$ and $X\in \Omega_M\otimes \Gamma(T_M)$.
The first equation says that $\rho_1$ can be viewed as a~bivector field, i.e.,
$
\langle\rho_1(-),-\rangle\in \Gamma\bigl(\wedge^2 T_M\bigr)$.
The second equation leads to
\[
\iota_X[\alpha,\beta] = \iota_X\mathcal{L}_{\rho(\alpha)}\beta -\iota_X\mathcal{L}_{\rho(\beta)}\alpha+\iota_X d(\iota_{\rho(\alpha)}\beta)
\]
for any $X$, hence forcing\footnote{Recall that in Proposition~\ref{prop:quadrel}, we showed a~Lie bracket structure can be defined on $\Gamma(V)$ based on the anchor $\rho_1$ and the torsion $\mu_2$.} $[\alpha,\beta]$ to be the Poisson bracket. Now the bivector anchor and the Poisson bracket together determine a~unique Poisson structure.
\end{proof}

From this point of view, Poisson sigma models~\cite{CF00, CFT, Kont97} can be viewed as geometric deformations obtained from cotangent theories over a~two-dimensional closed surface $\Sigma$
\[
S_0[X,\alpha]=\int_\Sigma\left\langle\alpha, {\rm d}_\Sigma X+\sum_{n=0}^\infty \frac1{n!}l_{n}\bigl(X^{\otimes n}\bigr)\right\rangle,
\]
where $ X\in\Omega_\Sigma\otimes_{\mathbb R} \Omega_M\otimes \Gamma(T_M)$ and $\alpha\in\Omega_\Sigma\otimes_{\mathbb R} \Omega_M\otimes \Gamma\bigl(T^\vee_M\bigr)[1]$.
The new structures $\{\rho_n,\mu_{n+1}\}_{n=1}^\infty$ give rise to an interaction action
\[
S_{def}[X,\alpha]=\int_\Sigma \frac12\left\langle\alpha, \sum_{n=0}^\infty \frac1{n!}\rho_{n+1}\bigl(X^{\otimes n},\alpha\bigr)\right\rangle.
\]
Moreover, if one wants to maintain the $(-1)$-shifted symplectic structure (i.e., the admissible pairing from the Lie point of view) as well as the underlying geometry (i.e., the information encoded in the cotangent bundle $T^\vee_M$), the only BV theories one could build are Poisson sigma models.

For the interest of model building, it would be useful to also consider the geometric deformations of the $L_\infty$ structure on graded vector bundles. Simple examples include $T^\vee_M[m]$ among others. This particularly suits the differential graded formulations of BV theories as developed in~\cite{CRM1, CRM2}. We plan to explore those possibilities in future work.

\appendix
\section{Recursive solution in the split Abelian extensions}
\label{appA}

In this appendix, we show that the $L_\infty$ structures given in Proposition~\ref{prop:hdefn} and~\ref{prop:vdefn}, which form a~split Abelian extension of $\mathfrak h$ by $\mathfrak v$, can be solved recursively from the quadratic equation on the respective CE differentials. In Lemma~\ref{recursivelemma}, we give the explicit formula, and the discussion leading to Proposition~\ref{propsolvhstru} shows that such structure is unique up to a~``gauge'' condition. Then we show in Proposition~\ref{htpyeqnhstru} that any two such structures are homotopy equivalent. Finally, we attach some calculation results for the first couple of operations.

On $C^*(\mathfrak h)$, the quadratic relation
\begin{equation}
\frac12\bigl[{\rm d}^{\nabla},{\rm d}^{\nabla}\bigr]=0
\label{appeqn1}
\end{equation}
for ${\rm d}^{\nabla}$ decomposes as follows. For $n\geq2$,
\[
\tag{\ref{appeqn1}.n}-\bigl[\delta, {\rm d}^\nabla_{n}\bigr]=\frac12\sum_{1\leq k\leq n-1}\bigl[{\rm d}_k^\nabla, {\rm d}^\nabla_{n-k}\bigr].
\]

For consistency, we also have
\[
\tag{\ref{appeqn1}.1}
-\bigl[\delta, {\rm d}^\nabla_{1}\bigr]=0.
\]

For the set of equations $\{(\ref{appeqn1}.n)\}_{n=1}^\infty$, the left-hand side of each equation is $\delta$-exact. If the right-hand side is not $\delta$-closed, such equation can not be solved. We say that equation $(\ref{appeqn1}.n)$ is {\it compatible} if its right-hand side commutes with $\delta$.

\begin{Lemma}
Let $n$ be a~positive integer.
If equation $(\ref{appeqn1}.k)$ is solvable for each $k\leq n$, then equation $(\ref{appeqn1}.n+1)$ is compatible.
\end{Lemma}

\begin{proof}
Consider $\delta$ acting on the right-hand side of $(\ref{appeqn1}.n+1)$
\[
\sum_{1\leq k\leq n}\frac12\bigl[\delta,\bigl[{\rm d}_k^\nabla, {\rm d}^\nabla_{1+n-k}\bigr]\bigr]=\frac12\sum_{1\leq k\leq n} \bigl[\bigl[\delta,{\rm d}_k^\nabla\bigr], {\rm d}^\nabla_{1+n-k}\bigr] - \bigl[{\rm d}_k^\nabla, \bigl[\delta,{\rm d}^\nabla_{1+n-k}\bigr]\bigr].
\]
Using $(\ref{appeqn1}.k)$ for all $1\leq k\leq n$, we have that
\[
-\frac14\sum_{1\leq m<k\leq n} \bigl[\bigl[{\rm d}_m^\nabla,{\rm d}_{k-m}^\nabla\bigr], {\rm d}^\nabla_{1+n-k}\bigr]+\frac14\sum_{1\leq k<m\leq n} \bigl[{\rm d}_k^\nabla,\bigl[{\rm d}_{m-k}^\nabla, {\rm d}^\nabla_{1+n-m}\bigr]\bigr].
\]
The summation can be re-expressed into a~more symmetric form
\[
\sum_{s+t+l=n+1}-\frac14\bigl[\bigl[{\rm d}_s^\nabla,{\rm d}_t^\nabla\bigr], {\rm d}^\nabla_l\bigr] +\frac14\bigl[{\rm d}_s^\nabla,\bigl[{\rm d}_t^\nabla, {\rm d}^\nabla_l\bigr]\bigr].
\]
Note that for a~fixed (unordered) positive partition of $n+1$, the number of assignments of $s$, $t$, $l$ varies due to the symmetry of each partition. For example, if we have $n+1 = a+a+b\,(a\neq b)$, then we shall have in total three terms for $\bigl[\bigl[{\rm d}_s^\nabla,{\rm d}_t^\nabla\bigr], {\rm d}^\nabla_l\bigr]$, corresponding to $s=b$, $t=b$ and~${l=b}$, respectively. So we consider each case separately.
\begin{enumerate}
\itemsep=0pt
\item[(1)] Case $s=t=l$.
\item[(2)] Case exactly two of $s$, $t$, $l$ are equal.
\item[(3)] Case $s\neq t\neq l$.
\end{enumerate}
In each case, the contributions cancel among themselves due to the Jacobi identity of differential operators under composition.
\end{proof}

\begin{Lemma}\label{recursivelemma}
Assume that equations $\{(\ref{appeqn1}.k)\}_{k=1}^n$ are solvable by $\{{\rm d}^\nabla_{k}\}_{k=1}^{n}$, if equation $(\ref{appeqn1}.n+1)$ is compatible, then there exists a~derivation ${\rm d}^\nabla_{n+1}$ solving equation $(\ref{appeqn1}.n+1)$, under which the generators \smash{$\bigl\{y^i\bigr\}_{i=1}^{\dim (M)}$} of $C^*(\mathfrak h)$ have homogeneous images of $^{\mathfrak h}$weight $(n+1)$ given by
\[
{\rm d}^\nabla_{n+1}\colon\ y^i\mapsto - \frac12 \sum_{1\leq k\leq n}\delta^{-1}\bigl( \bigl[{\rm d}^{\nabla}_k, {\rm d}^{\nabla}_{n+1-k}\bigr]\bigl(y^i\bigr)\bigr), \qquad \forall i.
\]
\end{Lemma}

\begin{proof} A~choice of Levi-Civita connection $\nabla$ induces a~connection on $C^*(\mathfrak h)$, which is again denoted by $\nabla$. It is known that ${\rm d}^\nabla_1=\nabla$, which is locally written (in compatible coordinates) as
\[ {\rm d}^\nabla_1= d +{\rm d}x^j\bigl(\nabla_j y^i\bigr)_k y^k\frac\partial{\partial y^i}.\]
For $k\geq2$, each ${\rm d}^\nabla_k$ is $\Omega_M$-linear, and is entirely determined by the image of $y^i$ for each $i$. We claim that
\[{\rm d}^\nabla_{n+1}\bigl(y^i\bigr) = -\frac12 \sum_{1\leq k\leq n}\delta^{-1}\bigl( \bigl[{\rm d}^{\nabla}_k, {\rm d}^{\nabla}_{n+1-k}\bigr]\bigl(y^i\bigr)\bigr)\]
provides a~solution for ${\rm d}^\nabla_{n+1}\bigl(y^i\bigr)$, $\forall i$, and hence a~solution for ${\rm d}^\nabla_{n+1}$.

Let $\delta$ act on both sides of the above relation.
The left-hand side is precisely $-\bigl[\delta, {\rm d}^\nabla_{n+1}\bigr]\bigl(y^i\bigr)$.
The right-hand side is
\[\frac12\sum_{1\leq k\leq n-1} \delta \delta^{-1} \left(\bigl[{\rm d}_k^\nabla, {\rm d}^\nabla_{n-k}\bigr]\bigl(y^i\bigr) \right) \equiv \frac12 \sum_{1\leq k\leq n-1} \bigl(id - \delta^{-1} \delta\bigr)\left(\bigl[{\rm d}_k^\nabla, {\rm d}^\nabla_{n-k}\bigr]\bigl(y^i\bigr) \right). \]
But the last term vanishes, which follows from
\[ \sum_{1\leq k\leq n-1} \delta \left(\bigl[{\rm d}_k^\nabla, {\rm d}^\nabla_{n-k}\bigr]\bigl(y^i\bigr) \right)=\sum_{1\leq k\leq n-1} \bigl[\delta,\bigl[{\rm d}_k^\nabla, {\rm d}^\nabla_{n-k}\bigr]\bigr]\bigl(y^i\bigr) + \bigl[{\rm d}_k^\nabla, {\rm d}^\nabla_{n-k}\bigr]\bigl(\delta y^i\bigr)\]
and the compatibility assumption.
\end{proof}

In the following discussion, we shall denote by $\delta^{-1}\bigl({\rm d}^\nabla_n\bigr)$ the operator
\[\sum_{i=1}^{\dim (M)} \delta^{-1} \bigl({\rm d}^\nabla_n\bigl(y^i\bigr)\bigr)\frac\partial{\partial y^i}\in \Gamma\bigl( \operatorname{Sym}^{n}\bigl(T_M^\vee \bigr)\otimes T_M \bigr)\subset C^*(\mathfrak h)\otimes \mathfrak h[1]\]
 for $n\geq2$.

\begin{Proposition}\label{propsolvhstru}
The differential \smash{${\rm d}^\nabla\equiv \delta+\nabla+\sum_{n\geq2} {\rm d}^\nabla_n$} can be solved recursively and unique\-ly from the equation ${\rm d}^\nabla\circ {\rm d}^\nabla=0$ for any fixed Levi-Civita connection $\nabla$, subject to the gauge condition $\delta^{-1} \bigl({\rm d}^{\nabla}_n\bigr) = 0$ for all $n\geq2$.
\end{Proposition}

\begin{proof}
The induction starts with the relation for ${\rm d}^\nabla_1\equiv \nabla$, which states the torsion-free condition in the initial step. The construction for ${\rm d}^\nabla_{n+1}\bigl(y^i\bigr)$ uses a~gauge choice that $\delta^{-1} {\rm d}^{\nabla}_n\bigl(y^i\bigr) = 0$ for all $n\geq2$, $i=1,\dots,\dim (M)$. Indeed, any solution \smash{$\tilde {\rm d}^\nabla_{n+1}\bigl(y^i\bigr)$} satisfying the same gauge condition and the same initial condition $\tilde {\rm d}^\nabla_1\equiv \nabla$ has to be given by Lemma~\ref{recursivelemma}, since the operator $\delta$ is invertible on images of $\delta^{-1}$, and so the relation
$\delta\bigl( \tilde {\rm d}^\nabla_{n+1}\bigl(y^i\bigr)\bigr) = \delta \bigl({\rm d}^\nabla_{n+1}\bigl(y^i\bigr)\bigr)$, $ i=1,\dots, \dim (M)$
enforced by equation $(\ref{appeqn1}.n+1)$
implies that \smash{$ \tilde {\rm d}^\nabla_{n+1}= {\rm d}^\nabla_{n+1}$}.
\end{proof}

Indeed, the $\delta^{-1}$ gauge need not to be the unique choice. Without much effort, one obtains the following generalization.
\begin{Corollary}\label{gaugechoice}
For any \smash{$\phi\in \Gamma\bigl( \widehat{\operatorname{Sym}}^{\geq3}\bigl(T_M^\vee \bigr)\otimes T_M \bigr)$}, and for any Levi-Civita connection $\nabla$ on~$T_M$, there exists a~unique differential ${\rm d}^{\nabla,\phi}$ satisfying the quadratic equation~\eqref{appeqn1} such that
\[
{\rm d}^{\nabla,\phi}\equiv \delta+\nabla+\sum_{n\geq2} {\rm d}^{\nabla,\phi},
\]
and $\delta^{-1} \bigl({\rm d}^{\nabla,\phi}_{\geq2}\bigr) = \phi$.
\end{Corollary}
\begin{proof}
Consider the recursive solution
\[{\rm d}^{\nabla,\phi}_{n+1}\bigl( y^i\bigr)= - \frac12 \sum_{1\leq k\leq n}\delta^{-1}\bigl( \bigl[{\rm d}^{\nabla,\phi}_k, {\rm d}^{\nabla,\phi}_{n+1-k}\bigr]\bigl(y^i\bigr)\bigr)+ \delta \phi_{n+2}\bigl( y^i\bigr), \]
for $i\!=\!1,\dots,\dim (M)$, where $\phi_{n+2}$ is the homogeneous component of $\phi$ in $ {\Gamma\bigl( \operatorname{Sym}^{n+2}\bigl(T_M^\vee \bigr)\otimes T_M \bigr)}$.
The uniqueness of such a~solution is established through induction on $n$,observing that each~\smash{${\rm d}^{\nabla,\phi}_{n+1}$} is uniquely determined by a~set of equations
\smash{$\delta\bigl({\rm d}^{\nabla,\phi}_{n+1}\bigr)=\dots, \delta^{-1}\bigl({\rm d}^{\nabla,\phi}_{n+1}\bigr)=\dots$},
where the uniqueness of the right-hand sides follows from the inductive hypothesis.
\end{proof}

\begin{Lemma}\label{familylem}
Given a~family of connections $\nabla^t$ on $T_M$ and a~family of gauge conditions $\phi^t\in \smash{\Gamma\bigl( \widehat{\operatorname{Sym}}^{\geq3}\bigl(T_M^\vee \bigr)\otimes T_M \bigr)}$, both parameterised by the standard $n$-simplex $\Delta^n$, there exists a~unique flat $\Delta^n$-family of CE differentials ${\rm d}(t)$ on $C^*(\mathfrak h)$, such that ${\rm d}_1(t) = \!\nabla^t$ and \smash{${\delta^{-1}(\sum_{m\geq2}{\rm d}_m(t)) \!= \phi^t}$}.
\end{Lemma}
\begin{proof}
Consider the parameterised differential
\[d(t) \equiv \delta+\sum_{m\geq1}{\rm d}_m(t)= \delta + {\rm d}_1^{\nabla^t} +\cdots\]
and the gauge condition
$\delta^{-1}(d(t)) = \phi^t$.
By previous results, $d(t)$ exists uniquely.
It remains to construct a~flat connection ${\rm d}_t+F(t)$, of which $d(t)$ is a~flat section, with
\[F(t)\equiv \sum_{m\geq2} F_m(t), \qquad F_m(t) \in \Omega_{\Delta^n}\otimes_{\mathbb R}\Gamma\bigl( \operatorname{Sym}^m\bigl(T_M^\vee \bigr)\otimes T_M \bigr).\] A~recursive computation can be done by the flat section equation
\[[{\rm d}_t+\sum_{k\geq2} F_k(t), \delta+ {\rm d}_1(t)+\sum_{m\geq 2} {\rm d}_m(t)]=0,\]
and the solution is unique (considering that $\delta$ is invertible on zero-forms over $M$).
In particular,
\[
F_2(t) = -\delta^{-1}\bigl(\bigl[{\rm d}_t, \nabla^t \bigr]\bigr)={\rm d}t^i\wedge \delta^{-1}\left( \frac{\rm d}{{\rm d}t^i} \nabla^t\right).
\]
Then one needs to check the flatness
\[
\bigg[{\rm d}_t, \sum_{m\geq2}F_m(t)\bigg]+\frac12\bigg[\sum_{m\geq2}F_m(t),\sum_{k\geq2}F_k(t)\bigg] =0,
\]
which is equivalent to the equation
\begin{gather*}
\bigg[\delta,\big[{\rm d}_t, \sum_{m\geq2}F_m(t)\big]+\frac12\bigg[\sum_{m\geq2}F_m(t),\sum_{k\geq2}F_k(t)\bigg] \bigg] =0\\
\qquad= -\bigg[{\rm d}_t+\sum_{k\geq2}F_k(t), \bigg[\delta, \sum_{m\geq2}F_m(t)\bigg]\bigg] \\
\qquad= \bigg[{\rm d}_t+\sum_{k\geq2}F_k(t), \bigg[\sum_{l\geq1} {\rm d}_l(t), {\rm d}_t+\sum_{m\geq2}F_m(t)\bigg]\bigg]\\
\qquad=-\bigg[\sum_{l\geq1} {\rm d}_l(t),\bigg[{\rm d}_t, \sum_{m\geq2}F_m(t)\bigg]+\frac12\bigg[\sum_{m\geq2}F_m(t),\sum_{k\geq2}F_k(t)\bigg] \bigg].
\end{gather*}
In homogeneous weight components, the equation can be solved recursively, as long as the initial condition
$[{\rm d}_t, F_2(t)] = 0$
is satisfied. The latter can be checked explicitly.
\end{proof}

\begin{Proposition}[homotopy among different $\delta^{-1}$ gauge conditions]\label{htpyeqnhstru}
Let ${\rm d}^\nabla$ be a~CE differen\-tial constructed in Proposition~{\rm\ref{propsolvhstru}} subject to the condition  \smash{$\delta^{-1} \bigl({\rm d}_{\geq2}^\nabla\bigr) = 0$}, and let $\tilde{{\rm d}}=\delta+\tilde{{\rm d}}_{\geq1}$ be~a~CE differential satisfying the $\delta^{-1}$ gauge condition
$\delta^{-1} (\tilde{{\rm d}}_{\geq2}) = \phi$ for any \smash{$ \phi\in \Gamma\bigl( \widehat{\operatorname{Sym}}^{\geq3}\bigl(T_M^\vee \bigr)\otimes T_M \bigr)$}. Then those $L_\infty$ structures ${\rm d}^\nabla$ and $\tilde{{\rm d}}$ are homotopic for arbitrary~${n\geq0}$.
\end{Proposition}
\begin{proof}
Noticing that the space of connections and the space of gauge fixing conditions are both contractible, the result follows from Lemma~\ref{familylem}.\end{proof}

For the differential on $C^*(\mathfrak h\oplus \mathfrak v)$ encoding the $L_\infty$ structure of the semi-direct product, we also need to compute ${\rm d}^{\nabla'}\equiv {\rm d}^{\nabla'}_1+\sum_{n\geq2} {\rm d}^{\nabla'}_n$, where ${\rm d}^{\nabla'}_1$, upon acting on $\mathfrak v^\vee[-1]$, is the (de Rham extended) dual connection coming from a~chosen affine connection $\nabla'$ on vector bundle $V$. Now one can show that by picking a~local frame \smash{$\{v^a\}_{a=1}^{\mathrm{rk}(V)}$} for $V^\vee$, ${\rm d}^{\nabla'}$ can be solved uniquely from the recursive formula
\begin{gather*}
{\rm d}^{\nabla'}_{n+1}\colon\ v^a\mapsto - \frac12 \sum_{1\leq k\leq n}\delta^{-1}\bigl( \bigl[{\rm d}^{\nabla'}_k, {\rm d}^{\nabla'}_{n+1-k}\bigr]\bigl(v^a\bigr)\bigr)
 - \sum_{1\leq k\leq n-1}\delta^{-1}\bigl( {\rm d}^{\nabla}_k\circ {\rm d}^{\nabla'}_{n+1-k}\bigl(v^a\bigr)\bigr),\qquad \forall a,
\end{gather*}
up to a~gauge condition on $\delta^{-1} {\rm d}^{\nabla'}_n\bigl(v^a\bigr)$ for all $n\geq2$. Furthermore, there exists a~family version parallelising Lemma~\ref{familylem} and Proposition~\ref{htpyeqnhstru}, which we present without giving the detailed proof.

\begin{Proposition}[homotopy of split Abelian extensions]\label{htpyLinftystru}
Any two split Abelian extension structures on $\Omega_M\otimes(T_M[-1]\oplus V)$ are $n$-homotopic for arbitrary $n\geq0$.
\end{Proposition}

Local computations for lower components of ${\rm d}^{\nabla}$ and ${\rm d}^{\nabla'}$ are given here. The curvature 2-forms are locally $\frac12R_{ij}\,_m\,^n {\rm d}x^i\wedge {\rm d}x^j $ on $T_M$ and $\frac12R'_{ij}\,_a\,^b {\rm d}x^i\wedge {\rm d}x^j $ on $V$. When necessary, we use the asterisked notation for the dual Christoffel symbols and curvatures. Suppose the first gauge condition is given by
\[
\delta^{-1}\bigl({\rm d}^\nabla_2\bigr) = \frac1{3!}(\phi_3)^i\,_{jkl} y^j y^k y^l\frac\partial{\partial y^i},\qquad
\delta^{-1}\bigl({\rm d}^{\nabla'}_2\bigr) = \frac1{2!}(\psi_3)^a\,_{bjk} y^j y^k v^b\frac\partial{\partial v^a},
\]
where $\phi_3\in \Gamma\bigl( \operatorname{Sym}^{3}\bigl(T_M^\vee \bigr)\otimes T_M \bigr)$ and $\psi_3\in \Gamma\bigl( \operatorname{Sym}^{2}\bigl(T_M^\vee \bigr)\otimes V^\vee[-1]\otimes V[1] \bigr)$,
\begin{gather*}
 {\rm d}^\nabla_1\bigl(y^i\bigr) = \bigl( \Gamma^*\bigr)^i\,_k y^k,\qquad
 {\rm d}^\nabla_2\bigl(y^i\bigr) = -\frac13 {\rm d}x^mR^*_{km}\,^i\,_j y^j y^k+\frac1{2!}(\phi_3)^i\,_{mjk} {\rm d}x^m y^j y^k,\\
 {\rm d}^{\nabla'}_1\bigl(v^a\bigr) = \bigl(\Gamma'^*\bigr)^a\,_b v^b,\qquad
{\rm d}^{\nabla'}_2\bigl(v^a\bigr) = -\frac12 {\rm d}x^mR'^*_{im}\,^a\,_b v^by^i+(\psi_3)^a\,_{bjk} {\rm d}x^j y^k v^b.
\end{gather*}

Dually, for any $X,Y\in \Gamma(T_M)$ and $\alpha\in \Gamma(V)$,
\begin{gather*}
l_1(Y) = \nabla Y,\qquad
l'_1(\alpha) = \nabla' \alpha, \\
l_2(X,Y) \equiv l_2(y_j, y_k) X^j Y^k
= \left( \frac13R_{jmk}\,^i +\frac13 R_{kmj}\,^i +(\phi_3)_{mjk}\,^i \right) {\rm d}x^my_i X^j Y^k\\
\phantom{l_2(X,Y) \equiv l_2(y_j, y_k) X^j Y^k}{}= \frac13 \bigl(\iota_X R Y+\iota_Y R X\bigr)+(\phi_3)_{mjk}\,^i {\rm d}x^my_i X^j Y^k,\\
l'_2(Y, \alpha)= Y^i \alpha ^a\left( \frac12R'_{im}\,_a\,^b + (\psi_3)_{aim}\,^b \right){\rm d}x^m v_b
= \frac12\iota_Y R'\alpha+ (\psi_3)_{aim}\,^b {\rm d}x^m v_bY^i \alpha ^a.
\end{gather*}
To recover the results in $\delta^{-1}$ gauge, one can simply set $\phi_3=\psi_3=0$.

\section[Lower brackets from {\rm d}\^{}rho and {\rm d}\^{}mu]{Lower brackets from $\boldsymbol{ {\rm d}^\rho}$ and $\boldsymbol{{\rm d}^\mu}$}
\label{appB}

The recursion starts with $\rho_1$ and $\mu_2$. Defining
$\rho_1(v_a)\equiv (\rho)_a\,^iy_i$, $ {\rm d}^\rho_{1} \bigl(y^i\bigr)\equiv \bigl(\rho^*\bigr)^i\,_a\cdot v^a$,
one observes that $(\rho)^i\,_a$ is an odd transformation since the whole structure is $(\Omega_M, d)$-linear, (as opposed to $\mathbb R$-linear or $C^\infty_M$-linear). Like the local connection form, we have that \smash{$(\rho)_a\,^i=-\bigl(\rho^*\bigr)^i\,_a$}. We shall neglect the star whenever it is obvious from the context. One helpful way to understand this is to consider the superconnection $\bigl(\begin{smallmatrix}\nabla & \rho \\0&\nabla' \end{smallmatrix}\bigr)$ on $T_M\oplus V[1]$, where the dual of $\rho_1$ operation is determined by the dual connection.
Solving the equation
\smash{$-\bigl[\delta,{\rm d}^\rho_{1+1}\bigr]\bigl(y^i\bigr)= \bigl[{\rm d}^{\nabla}_1, {\rm d}^\rho_{1}\bigr]\bigl(y^i\bigr)+ {\rm d}^{\nabla'}_1\bigl({\rm d}^\rho_{1}\bigl(y^i\bigr)\bigr) $}
and dualizing the result,
one obtains
$\rho_{1+1}(X,\alpha) = -\rho(\nabla'_X(\alpha)) +\nabla_X(\rho(\alpha))$.

For the next order operation,
\[-\bigl[\delta,{\rm d}^\rho_{1+2}\bigr]\bigl(y^i\bigr)= \bigl[{\rm d}^{\nabla}_1, {\rm d}^\rho_{2}\bigr]\bigl(y^i\bigr)+\bigl[{\rm d}^{\nabla}_2, {\rm d}^\rho_{1}\bigr]\bigl(y^i\bigr)+ {\rm d}^{\nabla'}_2\bigl({\rm d}^\rho_{1}\bigl(y^i\bigr)\bigr)+{\rm d}^{\nabla'}_1\bigl({\rm d}^\rho_{2}\bigl(y^i\bigr)\bigr), \]
gives that
\begin{align*}
{\rm d}^\rho_{1+2}\bigl(y^i\bigr)={}&-\frac12 y^j y^k v^a\bigl( \partial_k(\rho_2)^i\,_{ja}+(\rho_2)^i\,_{na}\bigl(\Gamma^*\bigr)^n\,_{jk} +(\rho_2)^i\,_{jb}\bigl(\Gamma'^*\bigr)^b\,_{ka}\\
& - \bigl(\Gamma^*\bigr)^i\,_{kn} (\rho_2)^n_{ja}+\frac13 \bigl(R^*\bigr)_{nj}\,^i\,_k \rho^n\,_a-(\phi_3)^i\,_{jnk} \rho^n\,_a +(\rho_1)^i\,_{b} (\psi_3)^b\,_{ajk}\bigr).
\end{align*}
The existence of symmetric product $y^j$ and $y^k$ will automatically symmetrise the corresponding indices. So after taking duals,
\begin{align*}
\rho(y_j, y_k, v_a)={}&-\nabla_{\{k} \rho_2(y_{j\}}, v_a)+\rho_2(\nabla_{\{j}y_{k\}}, v_a)
+\rho_2(y_{\{j}, \nabla'_{k\}}v_a)
-\frac13 R_{n\{j}\,_{k\}}\,^i y_i \rho_a\,^n\\
&-(\phi_3)_{jnk}\,^i \rho_a\,^ny_i +\rho_{b}\,^i (\psi_3)_{ajk}\,^by_i,
\end{align*}
where we have used the notation $\{-,-,\cdots\}$ to denote the symmetrization with respect to the corresponding indices.

In the computation of the first Atiyah--Chern class in Section~\ref{ss:modular}, the following result (only works in $\delta^{-1}$ gauge) is helpful:
\begin{gather}
\nabla_{k} \rho_2(y_{j}, v_a)-\rho_2(\nabla_{k}y_{j}, v_a) -\rho_2(y_{j}, \nabla'_{k}v_a)\nonumber\\
\qquad= \nabla_{\{k} \rho_2(y_{j\}}, v_a)-\rho_2(\nabla_{\{k}y_{j\}}, v_a) -\rho_2(y_{\{j}, \nabla'_{k\}}v_a)+ \nabla_{[k} \rho_2(y_{j]}, v_a)\nonumber\\
\phantom{\qquad=}{}-\rho_2(\nabla_{[k}y_{j]}, v_a) -\rho_2(y_{[j}, \nabla'_{k]}v_a)\nonumber\\
\qquad= -\rho_{jka}\,^i y_i-\frac13 R_{n\{j}\,_{k\}}\,^i y_i \rho_a\,^n-\frac12 R'_{kja}\,^c\rho_c\,^i y_i+ \frac12 \rho_a\,^n R_{kjn}\,^i y_i\nonumber\\
\qquad=-\rho(y_j, y_k, v_a)+\frac13 R(\rho(v_a), y_j)y_k -\frac23 R(\rho(v_a), y_k)y_j -\frac12\rho( R'_{kj}v_a). \label{compmodeqn2}
\end{gather}

For ${\rm d}^\mu$, the first nontrivial recursive computation gives
\[-\bigl[\delta,{\rm d}^\mu_{2+1}\bigr]\bigl(v^a\bigr)= \bigl[{\rm d}^{\nabla'}_1, {\rm d}^\mu_{2}\bigr]\bigl(v^a\bigr)+ {\rm d}^{\rho}_1\bigl({\rm d}^{\nabla'}_{2}\bigl(v^a\bigr)\bigr),\]
so
\begin{align*}
{\rm d}^\mu_{2+1}\bigl(v^a\bigr)={}&-\frac12 y^kv^b\wedge v^c \bigl(\partial_k\bigl(\mu^a_{bc}\bigr)+\mu^a_{dc} \bigl(\Gamma'^*\bigr)_k\,^d\,_b + \mu^a_{bd} \bigl(\Gamma'^*\bigr)_k\,^d\,_c \\
&-\bigl(\Gamma'^*\bigr)_k\,^a\,_d \mu^{\rm d}_{bc} - R'^*_{ik}\,^a\,_b \bigl(\rho^*\bigr)^i\,_c +2(\psi_3)^a\,_{bki} \bigl(\rho^*\bigr)^i\,_c\bigr).
\end{align*}

Dually this gives
\begin{align*}
\mu_{2+1}(y_k, v_b, v_c)={}&-\nabla'_k\mu(v_b, v_c)+\mu(\nabla'_k v_b, v_c) + \mu( v_b, \nabla'_k v_c)
+ \frac12R'(\rho(v_c),y_k)v_b \\
&-\frac12R'(\rho(v_b),y_k)v_c
+(\psi_3)_{bki}\,^a(\rho)_c\,^iv_a-(\psi_3)_{cki}\,^a(\rho)_b\,^i v_a.
\end{align*}

\subsection*{Acknowledgements}

The author thanks Chenchang Zhu and the anonymous referees for many helpful suggestions. XC was partially supported by NSFC grant 11801588 and by Guangdong Natural Science Foundation grant 2018A030313273. XC thanks Kavli IPMU and the University of G\"ottingen, where part of the work was done.

\pdfbookmark[1]{References}{ref}
\LastPageEnding

\end{document}